\documentclass[11pt,reqno]{amsart}
\usepackage[utf8]{inputenc} 
\usepackage[T1]{fontenc}    
\usepackage{hyperref}
\usepackage{amsfonts}       
\usepackage{microtype}      

\usepackage{amscd,amssymb,amsmath,amsthm}
\usepackage[arrow,matrix]{xy}
\usepackage{graphicx}
\usepackage{multicol}
\usepackage{subfigure}
\usepackage{epstopdf}
\usepackage{color}
\newtheorem{thm}{Theorem}

\newtheorem{lemma}{Lemma}
\newtheorem{pro}{Proposition}

\numberwithin{equation}{section} 

\begin{document}

\vspace{0.5in}
\renewcommand{\bf}{\bfseries}
\renewcommand{\sc}{\scshape}
\vspace{0.5in}

\title[Discrete Dynamics of a Phytoplankton-Zooplankton Model ...]{Discrete Dynamics of a Phytoplankton-Zooplankton Model with Toxin-Mediated Interactions}

\author{Sobirjon Shoyimardonov}

\address{S.K. Shoyimardonov$^{a,b}$ \begin{itemize}
\item[$^a$] V.I.Romanovskiy Institute of Mathematics, 9, University str.,
Tashkent, 100174, Uzbekistan;
\item[$^b$] National University of Uzbekistan,  4, University str., 100174, Tashkent, Uzbekistan.
\end{itemize}}
\email{shoyimardonov@inbox.ru}







\begin{abstract}
We investigate the dynamics of a discrete phytoplankton-zooplankton model incorporating Holling type~III predation and Holling type~II toxin release. The existence and stability of positive fixed points are analyzed, and it is shown that when two such points, $E_1$ and $E_2$, exist, $E_2$ is always a saddle. A Neimark-Sacker bifurcation at $E_1$ is verified using the normal form method, indicating the emergence of closed invariant curves. This bifurcation implies that phytoplankton and zooplankton populations may exhibit sustained periodic oscillations, which could correspond to natural plankton bloom cycles. The global stability of the boundary equilibrium $(1,0)$ is also established. Numerical simulations are presented to illustrate and confirm the theoretical findings.
\end{abstract}

\maketitle

\section{Introduction}

Oceans play a vital role in regulating the Earth's climate, supporting biodiversity, and sustaining global fisheries. At the foundation of the marine food web, phytoplankton are responsible for producing nearly half of the world's oxygen through photosynthesis, while zooplankton act as primary consumers, transferring energy to higher trophic levels. Understanding the interaction dynamics between these populations is essential for predicting ecosystem responses to environmental changes such as nutrient enrichment, climate variability, and pollution. These interactions are inherently nonlinear and can give rise to complex dynamical behaviors, including oscillations, sudden regime shifts, and intricate spatio-temporal patterns. Numerous studies have been devoted to modeling marine ecosystems and analyzing such dynamics \cite{Chatt, Chen, Shang, Hong, Kome, Tian, RSH, RSHV}.

One ecologically important phenomenon emerging from these interactions is the occurrence of \textit{harmful algal blooms} (HABs), which have received considerable scientific attention in recent decades \cite{Feng, Mac, Xiao}. HAB species are typically classified into two main groups: toxin producers, which can contaminate seafood or kill fish, and high-biomass producers, which can deplete oxygen and cause mass mortalities of marine organisms. Some species exhibit characteristics of both categories. Toxin-producing phytoplankton (TPP) in particular play a significant role in bloom dynamics, as the toxins they release can reduce zooplankton grazing pressure, disrupt predator-prey balance, and even cause large-scale mortality events \cite{Chatt, Kirk}. Such toxin-mediated interactions are believed to be one of the key drivers of population succession and bloom termination, with important implications for human health, fisheries, and coastal economies.

Mathematical models provide a valuable framework for understanding these processes. In \cite{Chatt}, a continuous-time phytoplankton-zooplankton model was proposed to capture the effects of both predation and toxin release:
\begin{equation}\label{chat}
\left\{
\begin{aligned}
&\frac{dP}{dt} = bP\left(1 - \frac{P}{K}\right) - \alpha f(P)Z, \\
&\frac{dZ}{dt} = \beta f(P)Z - rZ - \theta g(P)Z,
\end{aligned}
\right.
\end{equation}
where \(P\) denotes the density of the TPP population and \(Z\) the density of the zooplankton population. Here, \(f(P)\) represents the predation functional response, and \(g(P)\) represents the distribution of toxic substances. The parameters \(b>0\) and \(K>0\) are the intrinsic growth rate and carrying capacity of phytoplankton; \(\alpha>0\) is the zooplankton predation rate; \(\beta>0\) is the conversion efficiency of consumed phytoplankton into zooplankton biomass; \(r>0\) is the natural mortality rate of zooplankton; and \(\theta>0\) quantifies toxin-induced zooplankton mortality.

While continuous-time models like~\eqref{chat} are well studied, many ecological processes, such as seasonal plankton blooms--naturally occur in discrete intervals, or are observed through periodic sampling. \textit{Discrete-time dynamical systems} are therefore an effective alternative, capable of capturing complex behaviors such as period-doubling cascades, quasiperiodicity, and chaos. These models can reveal rich dynamical structures that are often inaccessible to their continuous-time counterparts.

Within this framework, \textit{bifurcation theory} provides a powerful tool for exploring how qualitative changes in system behavior occur as parameters vary. Of particular interest in predator-prey systems is the \textit{Neimark-Sacker} (N-S) bifurcation, the discrete analogue of the Hopf bifurcation. This occurs when a fixed point loses stability through a pair of complex-conjugate eigenvalues crossing the unit circle, leading to the emergence of a closed invariant curve. Ecologically, such a transition corresponds to a shift from steady coexistence to sustained quasiperiodic oscillations, interpretable as persistent, irregular bloom cycles. Detecting and analyzing N-S bifurcations in toxin-mediated plankton models not only deepens the theoretical understanding of nonlinear ecosystem dynamics but also offers insight into the mechanisms driving real-world bloom variability.

Functional responses describe how predator consumption rates vary with prey density and are fundamental to modeling plankton-zooplankton interactions. Common forms include Holling type~I (linear), type~II (saturating), and type~III (sigmoidal), each capturing different ecological mechanisms such as handling time limitations or predator switching behavior. In system~\eqref{chat}, the specific forms of $f(P)$ and $g(P)$ play a decisive role in determining stability properties and bifurcation scenarios.

In the original work~\cite{Chatt}, where model~\eqref{chat} was introduced, nine distinct cases were proposed, corresponding to different combinations of functional forms for \(f(P)\) and \(g(P)\). In~\cite{Chen}, two of these cases (case~5 and case~7) were investigated in continuous time. The discrete-time counterparts of cases~1 through~8 were studied by Shoyimardonov in~\cite{SH, SH-2, SH-3, SH-4, SH-5}, where, for each case, the existence and classification of positive fixed points were established, the occurrence of Neimark–Sacker bifurcations was demonstrated, and the global dynamics were analyzed.

In the present work, we focus on the discrete-time version of model~\eqref{chat} with Holling type~III grazing and Holling type~II toxin release, corresponding to the remaining case in~\cite{Chatt}. The functional responses are given by
\[
f(P) = \frac{P^2}{\gamma^2 + P^2}, \quad g(P) = \frac{P}{\gamma + P},
\]
where \( \gamma > 0 \) is the half-saturation constant. This choice reflects a biologically realistic scenario in which predator consumption accelerates slowly at low prey densities and then saturates at high prey densities (Holling type~III), while toxin effects increase rapidly and approach saturation at higher prey levels (Holling type~II).

We denote
$$\overline{t}=bt, \, \overline{u}=\frac{P}{K}, \, \overline{v}=\frac{\alpha Z}{b}, \, \overline{\gamma}=\frac{\gamma}{K}, \, \overline{\beta}=\frac{\beta}{b}, \, \overline{r}=\frac{r}{b}, \, \overline{\theta}=\frac{\theta}{b}.$$

Then by dropping the overline sign at time $t\geq0$ we get

\begin{equation}\label{chenn}
\left\{\begin{aligned}
&\frac{du}{dt}=u(1-u)-\frac{u^2v}{\gamma^2+u^2}\\
&\frac{dv}{dt}=\frac{\beta u^2v}{\gamma^2+u^2}-rv-\frac{\theta uv}{\gamma+u}.\\
\end{aligned}\right.
\end{equation}

We consider the discrete-time version of the model \eqref{chat}, which takes the following form:
\begin{equation}\label{h12}
\begin{cases}
u^{(1)} = u(2 - u) - \frac{u^2v}{\gamma^2+u^2}, \\[2mm]
v^{(1)} = \frac{\beta u^2v}{\gamma^2+u^2}+ (1 - r)v - \dfrac{\theta uv}{\gamma + u},
\end{cases}
\end{equation}
where \( u \) and \( v \) represent the scaled densities of phytoplankton and zooplankton, respectively. Throughout this work, we assume that all parameters \( r, \beta, \theta, \gamma \) are positive.

In this paper, we investigate the dynamics of system~\eqref{h12}. Section~2 addresses the existence of positive fixed points, and Section~3 classifies their nature based on stability properties. Section~4 examines the occurrence of Neimark-Sacker bifurcations, while Section~5 explores the global dynamics of the model. Numerical simulations supporting and illustrating the analytical results are given in Section~6.

\section{Existence of positive fixed points}

It is clear that \( (0,0) \) and \( (1,0) \) are fixed points of system~\eqref{h12}.
To determine the positive fixed points, we set \( v^{(1)} = v \) and obtain the condition:

\begin{equation}\label{beta}
\beta = \Psi(u) := \frac{[ru+\theta u + r\gamma](\gamma^2 + u^2)}{u^2(\gamma + u)}.
\end{equation}

Some basic properties of the function \( \Psi(u) \) are as follows:
\[
\lim_{u \to 0^+} \Psi(u) = +\infty, \qquad \Psi(1) = \frac{( r +\theta+ r\gamma)(\gamma^2 + 1)}{\gamma + 1}.
\]
and the derivative of \( \Psi(u) \) is:
\begin{equation*}
\Psi'(u) = \frac{\gamma \left( \theta u^3 - 2 \gamma(r +\theta) u^2 - \gamma^2(4r +\theta) u - 2r \gamma^3 \right)}{u^3 (u + \gamma)^2}.
\end{equation*}

We now denote

\begin{equation}\label{hx}
h(u)=\theta u^3 - 2 \gamma(r +\theta) u^2 - \gamma^2(4r +\theta) u - 2r \gamma^3
\end{equation}

\[
h(0)=-2r \gamma^3<0, \ \ h(1)=\theta(1-2\gamma-\gamma^2)-2r\gamma(\gamma+1)^2
\]

We compute the derivative of the function \( h(u) \):
\begin{equation}\label{hder}
h'(u) = 3\theta u^2 - 4\gamma(r + \theta) u - \gamma^2 (4r + \theta).
\end{equation}

The discriminant of the quadratic equation \eqref{hder} is given by
\[
D = 4\gamma^2\left(7\theta^2 + 20r\theta + 4r^2\right) > 0,
\]
which implies that \( h'(u) = 0 \) always has two distinct real roots:
\[
u_1 = \frac{2\gamma(r+\theta) - \gamma\sqrt{7\theta^2 + 20r\theta + 4r^2}}{3\theta} < 0, \quad
u_2 = \frac{2\gamma(r+\theta) + \gamma\sqrt{7\theta^2 + 20r\theta + 4r^2}}{3\theta} > 0.
\]

Since \( h(0) = -2r\gamma^3 < 0 \) and the function \( h(u) \) attains a local maximum at the point \( u_1 < 0 \), we conclude that the local minimum value \( h(u_2) \) is always negative. Therefore, the equation \( h(u) = 0 \) has a unique positive root, which we denote by \( \widehat{u} \).

\begin{thm}\label{thm1}
The operator~\eqref{h12} has positive fixed points according to the following cases:
\begin{itemize}
    \item[(i)] If \( \beta > \Psi(1) \) and \( (r, \gamma, \theta) \in R_i \) for some \( i \in \{1,2,3,4\} \), then there exists a \textbf{unique} positive fixed point, where
    \begin{align*}
        R_1 &= \left\{(r, \gamma, \theta) \in \mathbb{R}_{+}^3: 0 < \gamma < \sqrt{7} - 2,\quad 0 < \theta \leq \frac{4r\gamma(1+\gamma)}{3 - 4\gamma - \gamma^2} \right\}, \\
        R_2 &= \left\{(r, \gamma, \theta) \in \mathbb{R}_{+}^3: \gamma \geq \sqrt{7} - 2,\quad \theta > 0 \right\}, \\
        R_3 &= \left\{(r, \gamma, \theta) \in \mathbb{R}_{+}^3: 0 < \gamma < \sqrt{2} - 1,\quad \frac{4r\gamma(1+\gamma)}{3 - 4\gamma - \gamma^2} < \theta < \frac{2r\gamma(1 + \gamma)^2}{1 - 2\gamma - \gamma^2} \right\}, \\
        R_4 &= \left\{(r, \gamma, \theta) \in \mathbb{R}_{+}^3: \sqrt{2} - 1 \leq \gamma < \sqrt{7} - 2,\quad \theta > \frac{4r\gamma(1+\gamma)}{3 - 4\gamma - \gamma^2} \right\}.
    \end{align*}

    \item[(ii)] Let \( (r, \gamma, \theta) \in R_5 \), where
    \[
    R_5 = \left\{(r, \gamma, \theta) \in \mathbb{R}_{+}^3: 0 < \gamma < \sqrt{2} - 1,\quad \theta > \frac{2r\gamma(1+\gamma)^2}{1 - 2\gamma - \gamma^2} \right\}.
    \]
    Then:
    \begin{itemize}
        \item[(ii.a)] If \( \beta > \Psi(1) \) or \( \beta = \Psi(\widehat{u}) \), then there exists a \textbf{unique} positive fixed point.
        \item[(ii.b)] If \( \Psi(\widehat{u}) < \beta < \Psi(1) \), then there exist \textbf{two} positive fixed points.
    \end{itemize}
\end{itemize}
\end{thm}

\begin{figure}
  \centering
  \includegraphics[width=1\textwidth]{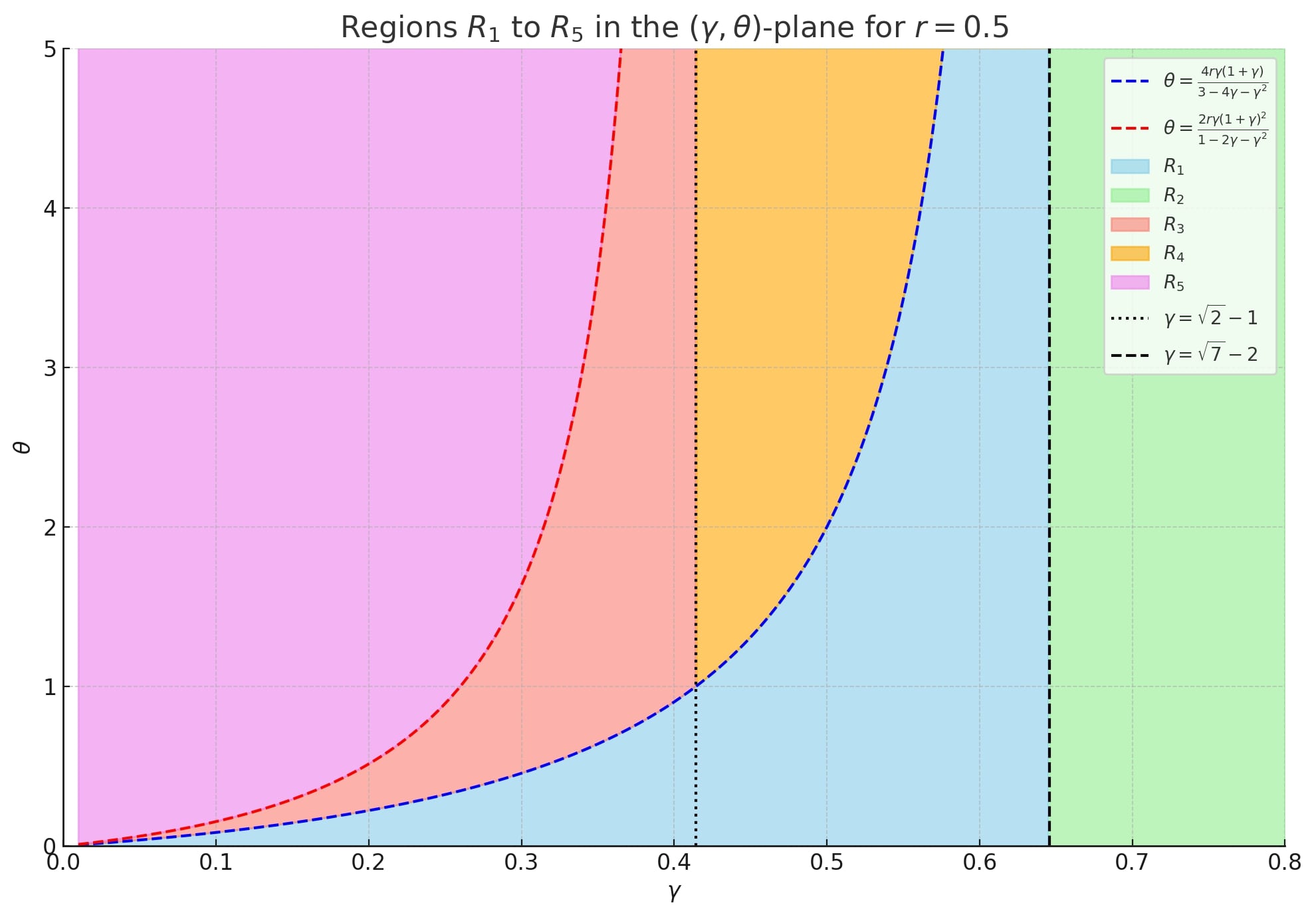}\\
   \caption{
    Partition of the $(\gamma, \beta)$-parameter space into five regions $R_1$ through $R_{5}$ corresponding to different numbers and configurations of positive fixed points of the system~\eqref{h12} for fixed $r = 0.5$.
    }\label{figr}
\end{figure}

\begin{proof} From the fixed point equation \( u^{(1)} = u \), we obtain the corresponding value of \( v \) as
\begin{equation}\label{v}
v = \frac{(1 - u)(\gamma^2 + u^2)}{u}.
\end{equation}
Therefore, we seek solutions of equation~\eqref{beta} that lie in the interval \( (0,1) \).

It is evident that if \( u_2 \geq 1 \), or if \( u_2 < 1 \) and \( h(1) \leq 0 \), then the equation \( h(u) = 0 \) has no solution in the interval \( (0,1) \), which implies that \( h(u) < 0 \) for all \( u \in (0,1) \). Consequently, the function \( \Psi(u) \) is strictly decreasing on \( (0,1) \), and the operator~\eqref{h12} admits a unique positive fixed point for every \( \beta > \Psi(1) \).

On the other hand, if \( u_2 < 1 \) and \( h(1) > 0 \), then the equation \( h(u) = 0 \) has a unique solution \( \widehat{u} \in (0,1) \). In this case, the function \( \Psi(u) \) is strictly decreasing on \( (0, \widehat{u}) \) and strictly increasing on \( (\widehat{u}, 1) \), so \( \Psi(\widehat{u}) \) is the global minimum of \( \Psi \) on \( (0,1) \). Therefore, the operator~\eqref{h12} has:
\begin{itemize}
    \item a unique positive fixed point if \( \beta > \Psi(1) \) or \( \beta = \Psi(\widehat{u}) \),
    \item exactly two positive fixed points if \( \Psi(\widehat{u}) < \beta < \Psi(1) \).
\end{itemize}

We now derive the parameter conditions corresponding to each of these cases.

\textbf{Case 1:} \( u_2 \geq 1 \). Solving this inequality yields:
\[
(\gamma^2 + 4\gamma - 3)\theta + 4r\gamma(1 + \gamma) \geq 0.
\]
The solution set is given by:
\[
0 < \gamma < \sqrt{7} - 2, \quad 0 < \theta \leq \frac{4r\gamma(1 + \gamma)}{3 - 4\gamma - \gamma^2},
\]
or
\[
\gamma \geq \sqrt{7} - 2, \quad \theta > 0,
\]
where \( r > 0 \) is any positive parameter.

\textbf{Case 2:} \( u_2 < 1 \) and \( h(1) \leq 0 \). Note that
\[
h(1) = \theta(1 - 2\gamma - \gamma^2) - 2r\gamma(1 + \gamma)^2.
\]
Solving the inequalities \( u_2 < 1 \) and \( h(1) \leq 0 \), we obtain the conditions:
\[
0 < \gamma < \sqrt{2} - 1, \quad \frac{4r\gamma(1 + \gamma)}{3 - 4\gamma - \gamma^2} < \theta \leq \frac{2r\gamma(1 + \gamma)^2}{1 - 2\gamma - \gamma^2},
\]
or
\[
\sqrt{2} - 1 \leq \gamma < \sqrt{7} - 2, \quad \theta > \frac{4r\gamma(1 + \gamma)}{3 - 4\gamma - \gamma^2}.
\]

\textbf{Case 3:} \( u_2 < 1 \) and \( h(1) > 0 \). Then the equation \( h(u) = 0 \) has a unique solution \( \widehat{u} \in (0,1) \), and we obtain:
\[
0 < \gamma < \sqrt{2} - 1, \quad \theta > \frac{2r\gamma(1 + \gamma)^2}{1 - 2\gamma - \gamma^2}.
\]
Figure~\ref{figr} illustrates all the parameter regions in the \( (\gamma, \theta) \)-plane corresponding to the specific value \( r = 0.5 \). This completes the proof.
\end{proof}

\section{Type of fixed points}

The following lemma is useful for the analysis of the fixed points.

\begin{lemma}[Lemma 2.1, \cite{Cheng}]\label{lem1}
Let \( F(\lambda) = \lambda^2 + B\lambda + C \), where \( B \) and \( C \) are real constants. Suppose \( \lambda_1 \) and \( \lambda_2 \) are the roots of the equation \( F(\lambda) = 0 \). Then the following statements hold:
\begin{itemize}
    \item[(i)] If \( F(1) > 0 \), then:
    \begin{itemize}
        \item[(i.1)] \( |\lambda_1| < 1 \) and \( |\lambda_2| < 1 \) if and only if \( F(-1) > 0 \) and \( C < 1 \);
        \item[(i.2)] \( \lambda_1 = -1 \) and \( \lambda_2 \neq -1 \) if and only if \( F(-1) = 0 \) and \( B \neq 2 \);
        \item[(i.3)] \( |\lambda_1| < 1 \), \( |\lambda_2| > 1 \) if and only if \( F(-1) < 0 \);
        \item[(i.4)] \( |\lambda_1| > 1 \) and \( |\lambda_2| > 1 \) if and only if \( F(-1) > 0 \) and \( C > 1 \);
        \item[(i.5)] \( \lambda_1 \) and \( \lambda_2 \) are complex conjugates with \( |\lambda_1| = |\lambda_2| = 1 \) if and only if \( -2 < B < 2 \) and \( C = 1 \);
        \item[(i.6)] \( \lambda_1 = \lambda_2 = -1 \) if and only if \( F(-1) = 0 \) and \( B = 2 \).
    \end{itemize}

    \item[(ii)] If \( F(1) = 0 \), i.e., \( 1 \) is a root of \( F(\lambda) = 0 \), then the other root \( \lambda \) satisfies:
    \[
    |\lambda| \begin{cases}
    < 1, & \text{if } |C| < 1, \\
    = 1, & \text{if } |C| = 1, \\
    > 1, & \text{if } |C| > 1.
    \end{cases}
    \]

    \item[(iii)] If \( F(1) < 0 \), then the equation \( F(\lambda) = 0 \) has one root in the interval \( (1, \infty) \). Moreover:
    \begin{itemize}
        \item[(iii.1)] The other root \( \lambda \) satisfies \( \lambda \leq -1 \) if and only if \( F(-1) \leq 0 \);
        \item[(iii.2)] The other root \( \lambda \) satisfies \( -1 < \lambda < 1 \) if and only if \( F(-1) > 0 \).
    \end{itemize}
\end{itemize}
\end{lemma}

\begin{pro}\label{prop1}
For the fixed points \( (0, 0) \) and \( (1, 0) \) of system~\eqref{h12}, the following statements hold:
\[
(0, 0) =
\begin{cases}
\text{nonhyperbolic}, & \text{if } r = 2, \\[2mm]
\text{saddle}, & \text{if } 0 < r < 2, \\[2mm]
\text{repelling}, & \text{if } r > 2,
\end{cases}
\]
\[
(1, 0) =
\begin{cases}
\text{nonhyperbolic}, & \text{if } \theta = \left(\frac{\beta}{1+\gamma^2}-r\right)(1+\gamma) \text{ or } \theta= \left(\frac{\beta}{1+\gamma^2}-r+2\right)(1+\gamma), \\[2mm]
\text{attractive}, & \text{if } \left(\frac{\beta}{1+\gamma^2}-r\right)(1+\gamma) < \theta < \left(\frac{\beta}{1+\gamma^2}-r+2\right)(1+\gamma), \\[2mm]
\text{saddle}, & \text{otherwise}.
\end{cases}
\]
\end{pro}

\begin{proof}
The Jacobian matrix of the operator \eqref{h12} is given by
\begin{equation}\label{jac01}
J(u,v) =
\begin{bmatrix}
2 - 2u-\frac{2\gamma^2 uv}{(\gamma^2+u^2)^2}  & -\frac{u^2}{\gamma^2+u^2} \\
\gamma v \left( \frac{2\beta \gamma u }{(\gamma^2+u^2)^2} - \frac{\theta}{(\gamma+u)^2} \right) & 1 - r + \frac{ \beta u^2}{ \gamma^2+u^2} - \frac{\theta u }{ \gamma+u}
\end{bmatrix}.
\end{equation}

Evaluating the Jacobian at the fixed point \( (0,0) \), we obtain
\[
J(0,0) =
\begin{bmatrix}
2 & 0 \\
0 & 1 - r
\end{bmatrix},
\]
whose eigenvalues are \( \lambda_1 = 2 \) and \( \lambda_2 = 1 - r \). Since \( \lambda_1 = 2 > 1 \), the instability of the fixed point \( (0,0) \) is immediate.

Next, consider the fixed point \( (1,0) \). The Jacobian at this point becomes
\[
J(1,0) =
\begin{bmatrix}
0 & -\dfrac{1}{1+\gamma^2} \\
0 & 1 - r+\dfrac{\beta}{1+\gamma^2} - \dfrac{\theta}{1+\gamma}
\end{bmatrix},
\]
which is a lower triangular matrix. The eigenvalues are therefore
\[
\lambda_1 = 0, \quad \lambda_2 = 1 - r+\dfrac{\beta}{1+\gamma^2} - \dfrac{\theta}{1+\gamma} .
\]

Stability of the fixed point \( (1,0) \) depends on the condition \( |\lambda_2| < 1 \), while \( |\lambda_2| = 1 \) indicates a bifurcation threshold. These yield conditions on the parameters \( \beta, \theta, r, \gamma \), which determine the local dynamics near \( (1,0) \).

This completes the proof.
\end{proof}

Now we classify the types of positive fixed points. Let $E_1=(u_{1}, v_{1})$ and $E_2=(u_{2}, v_{2})$ be positive fixed points of the operator \eqref{h12} such that $u_{1} < u_{2}$, where $u_i$ is a solution of the equation $\theta=\Psi(u)$ in (0,1) and $v_i=\frac{(1-u_i)(\gamma^2+u_i^2)}{u_i}$ for any $i=1,2.$ Then using them we simplify the Jacobian of the operator \eqref{h12} at any positive fixed point:

\begin{equation}\label{jac}
J(u_i) =
\begin{bmatrix}
\frac{2u_i^2(1-u_i)}{\gamma^2+u_i^2} &-\frac{u_i^2}{\gamma^2+u_i^2} \\
\frac{\gamma (1-u_i)(\gamma^2+u_i^2) }{u_i}\cdot\left( \frac{2\beta \gamma u_i }{(\gamma^2+u_i^2)^2} - \frac{\theta}{(\gamma+u_i)^2} \right) &  1
\end{bmatrix}.
\end{equation}

Its characteristic polynomial is

\begin{equation}\label{F}
\mathcal{P}(\lambda,u_i)=\lambda^2-p(u_i)\lambda+q(u_i)
\end{equation}

where
\begin{align}
p(u_i) &= 1 + \frac{2u_i^2(1-u_i)}{\gamma^2+u_i^2}, \label{pu}\\
q(u_i) &= \frac{2u_i^2(1-u_i)}{\gamma^2+u_i^2}
+ \gamma u_i(1 - u_i) \left( \frac{2\beta \gamma u_i }{(\gamma^2+u_i^2)^2} - \frac{\theta}{(\gamma+u_i)^2} \right) \label{qu}
\end{align}

The following theorem classifies positive fixed points of the operator (\ref{h12}).

\begin{thm}\label{type}
Let $\mathcal{P}(\lambda,u_i)$ be defined as in~\eqref{F}, and let \( q(u_i) \) be defined as in~\eqref{qu}.
Consider the fixed points $E_i = (u_i, v_i)$ for $i = 1, 2$ of the operator~\eqref{h12}. Then the following statements hold:

\begin{itemize}
    \item[(i)] The type of the fixed point $E_1$ is determined by:
    \[
    E_1 =
    \begin{cases}
    \text{attracting}, & \text{if } q(u_1) < 1, \\
    \text{repelling}, & \text{if } q(u_1) > 1, \\
    \text{nonhyperbolic}, & \text{if } q(u_1) = 1.
    \end{cases}
    \]

    \item[(ii)] The fixed point $E_2,$ if it exists, is a saddle point.
\end{itemize}
\end{thm}

\begin{proof}
Let $\widehat{u} \in (0,1)$ be the unique solution of $h(u) = 0$ defined in~\eqref{hx}, and note that $\beta = \Psi(u_i)$ for all $i = 1,2$. Moreover, the function $h(u)$ satisfies
\[
h(u) < 0 \quad \text{for} \quad u \in (0, \widehat{u}), \qquad h(u) > 0 \quad \text{for} \quad u \in (\widehat{u}, 1).
\]
It is evident that
\[
u_1 \in (0, \widehat{u}), \qquad u_2 \in (\widehat{u}, 1).
\]

\begin{itemize}
    \item[(i)] We first show that $\mathcal{P}(1, u_1) > 0$ and $\mathcal{P}(1, u_2) < 0$.

    From~\eqref{F}, we have
    \[
    \mathcal{P}(1, u_i) = \gamma u_i(1 - u_i) \left( \frac{2\beta \gamma u_i }{(\gamma^2+u_i^2)^2} - \frac{\theta}{(\gamma+u_i)^2} \right).
    \]
    Using the identity $\beta = \Psi(u_i)$, this simplifies to
    \[
    \mathcal{P}(1, u_i) = -\frac{\gamma(1 - u_i) h(u_i)}{(\gamma + u_i)^2 (\gamma^2 + u_i^2)}.
    \]
    Since $h(u_1) < 0$ and $h(u_2) > 0$, it follows that $\mathcal{P}(1, u_1) > 0$ and $\mathcal{P}(1, u_2) < 0$.

    Furthermore, from $\mathcal{P}(1, u_1) > 0$, it directly follows that $\mathcal{P}(-1, u_1) > 0$. Hence, by Lemma~\ref{lem1}, this case is established.

    \item[(ii)] As shown above, $\mathcal{P}(1, u_2) < 0$, which implies that one eigenvalue of the Jacobian matrix at the fixed point \( E_2 \) lies in the interval \( (1, \infty) \). Now we show that $\mathcal{P}(-1, u_2) > 0$.

    Observe that
    \[
    \mathcal{P}(-1, u_2) = 2 + \frac{4u_2^2(1 - u_2)}{\gamma^2 + u_2^2} - \frac{\gamma(1 - u_2) h(u_2)}{(\gamma + u_2)^2 (\gamma^2 + u_2^2)}.
    \]
    Since $h(u_2) > 0$, we conclude that $\mathcal{P}(-1, u_2) > 0$ is equivalent to the inequality:
    \[
    h(u_2) < \frac{2(\gamma + u_2)^2(\gamma^2 + u_2^2)}{\gamma(1 - u_2)} + \frac{4u_2^2(\gamma + u_2)^2}{\gamma}.
    \]

    We define the following auxiliary function:
    \[
    g(x) = \frac{2(\gamma + x)^2(\gamma^2 + x^2)}{\gamma(1 - x)} + \frac{4x^2(\gamma + x)^2}{\gamma}.
    \]

    We now investigate the behavior of \( g(x) \). Clearly,
    \[
    \lim_{x \to 1^-} g(x) = +\infty, \qquad g(0) = 2\gamma^3 > 0.
    \]

    The first derivative is:
    \[
    g'(x) = \frac{2(\gamma + x)\left[x^2(8x^2 - 19x + 12) + \gamma x(4x^2 - 9x + 6) + (2 - x)\gamma^2 + \gamma^3 \right]}{\gamma(1 - x)^2}.
    \]
    It is evident that \( g'(x) > 0 \) for all \( x \in (0,1) \), which implies that \( g(x) \) is strictly increasing on \( (0,1) \).

    The second derivative is given by:
    \[
    g''(x) = \frac{4\left[a(x) + b(x)\gamma + c(x)\gamma^2 - 2\gamma^3 - \gamma^4\right]}{\gamma(1 - x)^3},
    \]
    where
    \[
    \begin{aligned}
    a(x) &= 12x^5 - 39x^4 + 44x^3 - 18x^2, \\
    b(x) &= 12x^4 - 38x^3 + 42x^2 - 18x, \\
    c(x) &= -6x^2 + 6x - 4.
    \end{aligned}
    \]
   By straightforward calculations, it can be verified that \( a(x) < 0 \), \( b(x) < 0 \), and \( c(x) < 0 \) for all \( x \in (0,1) \) (as shown in Fig.~\ref{fig1}). Consequently, we conclude that \( g''(x) > 0 \) on \( (0,1) \), and hence the function \( g(x) \) is convex on this interval.

   Since \( h(x) \) is also convex on \( (0,1) \) and \( h(0) = -2r\gamma^3 < 0 \), the graphs of \( h(x) \) and \( g(x) \) do not intersect in \( (0,1) \), which implies that \( h(x) < g(x) \) for all \( x \in (0,1) \). Hence, \( \mathcal{P}(-1, u_2) > 0 \), and by Lemma~\ref{lem1}, this case is also proved.
\end{itemize}

The proof is complete.

\begin{figure}
  \centering
  \includegraphics[width=0.9\textwidth]{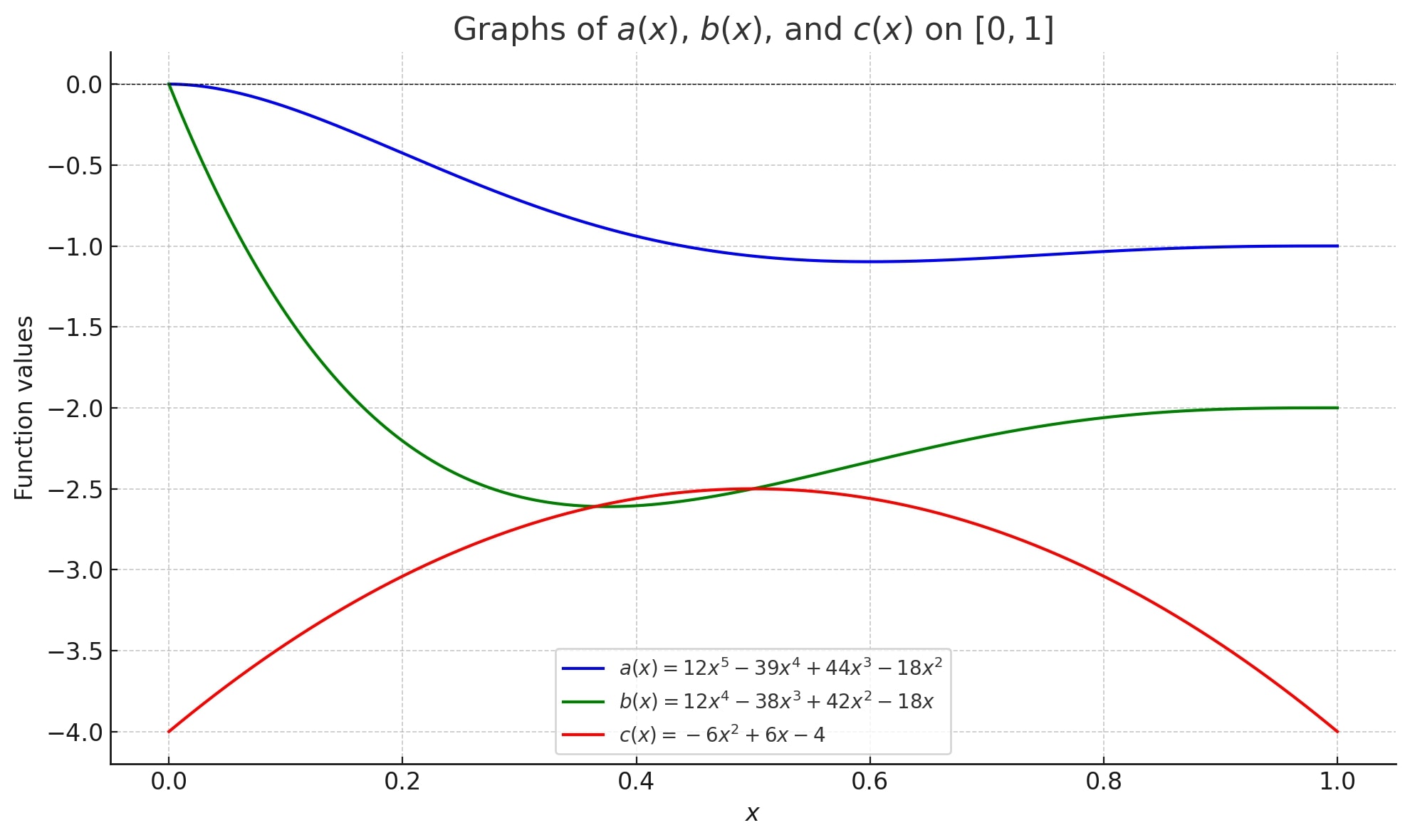}\\
   \caption{Graphs of $a(x),$ $b(x)$ and $c(x)$ on the interval [0,1].}\label{fig1}
\end{figure}

\end{proof}

\section{Neimark--Sacker Bifurcation analysis}

In this section, we determine the conditions under which a Neimark-Sacker bifurcation occurs at the positive fixed point \( E_1 = (u_1, v_1) \).

By Theorem~\ref{type} and Lemma~\ref{lem1}, the Jacobian matrix at \( E_1 \) possesses a pair of complex conjugate eigenvalues of modulus one when
\[
q(u_1) = 1,
\]
where \( q(u_i) \) is given in~\eqref{qu}.
We select \( \beta \) as the bifurcation parameter and denote by \( \beta_0 \) a solution of \( q(u_1) = 1 \).

We now prove that, as \( \beta \) varies in a sufficiently small neighborhood of \( \beta_0 \), the fixed point \( E_1 \) undergoes a Neimark-Sacker bifurcation. The proof proceeds as follows:

\textbf{Step~1.}
Introduce the change of variables
\[
x = u - u_1, \quad y = v - v_1,
\]
which shifts the fixed point \( E_1 \) to the origin.
Under this transformation, system~\eqref{h12} becomes
\begin{equation}\label{bif1}
\left\{
\begin{aligned}
x^{(1)} &= (x + u_1)\big(2 - x - u_1\big)
          - \frac{(x + u_1)^2 (y + v_1)}{\gamma^2 + (x + u_1)^2}
          - u_1, \\[2mm]
y^{(1)} &= (y + v_1) \left[
          \frac{\beta (x + u_1)^2}{\gamma^2 + (x + u_1)^2}
          - \frac{\theta (x + u_1)}{\gamma + x + u_1}
          + 1 - r
          \right]
          - v_1.
\end{aligned}
\right.
\end{equation}

\textbf{Step~2.}
Introduce a small perturbation \( \beta^* \) to the parameter \( \beta \), such that
\[
\beta = \beta_0 + \beta^*.
\]
The perturbed system can be written as
\begin{equation}\label{bif2}
\left\{
\begin{aligned}
x^{(1)} &= (x + u_1)\big(2 - x - u_1\big)
          - \frac{(x+u_1)^2 (y + v_1)}{\gamma^2 + (x + u_1)^2}
          - u_1, \\[2mm]
y^{(1)} &= (y + v_1) \left[
          \frac{(\beta_0 + \beta^*)(x + u_1)^2}{\gamma^2 + (x + u_1)^2}
          - \frac{\theta(x + u_1)}{\gamma + x + u_1}
          + 1 - r
          \right]
          - v_1.
\end{aligned}
\right.
\end{equation}

The Jacobian matrix of system~\eqref{bif2} at the origin \((0,0)\) is given by

\begin{equation}\label{jac2}
\mathcal{J}(0,0) =
\begin{bmatrix}
\frac{2u_1^2(1-u_1)}{\gamma^2+u_1^2} &-\frac{u_1^2}{\gamma^2+u_1^2} \\
\frac{\gamma (1-u_1)(\gamma^2+u_1^2) }{u_1}\cdot\left( \frac{2(\beta_0 + \beta^*) \gamma u_1 }{(\gamma^2+u_1^2)^2} - \frac{\theta}{(\gamma+u_1)^2} \right) &  1+\frac{\beta^*u_1^2}{\gamma^2+u_1^2}
\end{bmatrix}.
\end{equation}

The corresponding characteristic equation of the Jacobian at the origin is
\[
\lambda^2 - \nu(\beta^*)\,\lambda + \omega(\beta^*) = 0,
\]
where the trace and determinant are
\[
\nu(\beta^*) = \operatorname{Tr}(\mathcal{J})
= 1 + \frac{2u_1^2(1-u_1)}{\gamma^2+u_1^2}
  + \frac{\beta^* u_1^2}{\gamma^2+u_1^2}
= p(u_1) + \frac{\beta^* u_1^2}{\gamma^2+u_1^2},
\]
and, using \( q(u_1) = 1 \) at \( \beta = \beta_0 \), the determinant is
\[
\omega(\beta^*) = \det(\mathcal{J})
= 1 + \frac{2\beta^*\, u_1^2(1 - u_1)}{\gamma^2 + u_1^2}.
\]

Since \( \nu(0) = p(u_1) < 2 \) and \( \omega(0) = 1 \), the characteristic roots are complex conjugate with modulus close to one:
\begin{equation}\label{bif5}
\lambda_{1,2} = \frac{1}{2}\left[ \nu(\beta^*) \pm i \sqrt{4\omega(\beta^*) - \nu^2(\beta^*)} \right],
\end{equation}
with \(|\lambda_{1,2}| = \sqrt{\omega(\beta^*)}\).

Differentiating, we obtain:
\begin{equation}\label{bif7}
\frac{d|\lambda_{1,2}|}{d\beta^*}\Big|_{\beta^* = 0} = \frac{u_1^2(1 - u_1)}{\gamma^2 + u_1^2}> 0.
\end{equation}

Hence, the transversality condition
\[
\frac{d|\lambda_{1,2}|}{d\beta^*} \Big|_{\beta^* = 0} \neq 0
\]
is satisfied. Moreover, since \( 1 < \nu(0) < 2 \) and \(\omega(0) = 1 \), it follows that \( \lambda_{1,2}^m(0) \neq 1 \) for \( m = 1, 2, 3, 4 \), ensuring the non-resonance condition.

Therefore, all the classical conditions for the occurrence of a Neimark--Sacker bifurcation at the fixed point \( E_1 \) are satisfied.

\textbf{Step 3.} To derive the normal form of system~\eqref{bif2} at the critical parameter value \( \beta = \beta_0 \), we expand it into a third-order Taylor series around the fixed point \( (x, y) = (0, 0) \). This yields:
\begin{equation}\label{bif8}
\left\{
\begin{aligned}
x^{(1)} &= a_{10}x + a_{01}y + a_{20}x^2 + a_{11}xy + a_{02}y^2 + a_{30}x^3 + a_{21}x^2y + a_{12}xy^2 + a_{03}y^3 + \mathcal{O}(\rho^4), \\
y^{(1)} &= b_{10}x + b_{01}y + b_{20}x^2 + b_{11}xy + b_{02}y^2 + b_{30}x^3 + b_{21}x^2y + b_{12}xy^2 + b_{03}y^3 + \mathcal{O}(\rho^4),
\end{aligned}
\right.
\end{equation}
where \( \rho = \sqrt{x^2 + y^2} \), and the coefficients are given by:
\begin{align*}
&a_{10} = \frac{2u_1^2(1-u_1)}{\gamma^2+u_1^2},
&& a_{01} = -\frac{u_1^2}{\gamma^2+u_1^2}, \\
&a_{20} = \frac{\gamma^2u_1^2(3 - 5u_1) - u_1^5-\gamma^4}{u_1(\gamma^2 + u_1^2)^2},
&& a_{11} = -\frac{2\gamma^2u_1}{(\gamma^2 + u_1^2)^2}, \\
&a_{30} = \frac{4\gamma^2(1-u_1)(\gamma^2- u_1^2)}{(\gamma^2+ u_1^2)^3},
&& a_{21} = \frac{\gamma^2(3u_1^2-\gamma^2)}{(\gamma^2+ u_1^2)^3}, \\
&a_{02} = a_{03} = a_{12} = 0, \\[5pt]
&b_{10} = \frac{\gamma (1-u_1)(\gamma^2+u_1^2) }{u_1}\cdot\left( \frac{2\beta_0\gamma u_1 }{(\gamma^2+u_1^2)^2} - \frac{\theta}{(\gamma+u_1)^2} \right),
&& b_{01} = 1, \\
&b_{20} = \frac{\gamma (1-u_1)}{u_1}\cdot\left( \frac{\beta_0\gamma (\gamma^2-3u_1^2)}{(\gamma^2+u_1^2)^2} + \frac{\theta(\gamma^2+u_1^2)}{(\gamma+u_1)^3} \right), \\
&b_{11} = \gamma \left( \frac{2\beta_0\gamma u_1}{(\gamma^2 + u_1^2)^2} - \frac{\theta}{(\gamma + u_1)^2} \right), \\
&b_{30} = \frac{\gamma (1-u_1)(\gamma^2+u_1^2)}{u_1}\cdot\left( \frac{4\beta_0\gamma u_1 (u_1^2-\gamma^2)}{(\gamma^2+u_1^2)^4} - \frac{\theta}{(\gamma+u_1)^4} \right), \\
&b_{21} = \gamma \left( \frac{\beta_0\gamma(\gamma^2-3u_1^2)}{(\gamma^2 + u_1^2)^3} + \frac{\theta}{(\gamma + u_1)^3} \right), \\
&b_{02} = b_{03} = b_{12} = 0.
\end{align*}

The Jacobian matrix at the fixed point \( E_1 = (u_1, v_1) \) is
\[
J(E_1) =
\begin{bmatrix}
a_{10} & a_{01} \\
b_{10} & b_{01}
\end{bmatrix}.
\]
At the critical parameter value \( \beta = \beta_0 \), the condition \( q(u_1) = 1 \) holds, and the corresponding eigenvalues of \( J(E_1) \) are
\[
\lambda_{1,2} = \frac{1 + a_{10} \mp i\,\alpha}{2},
\quad \text{where} \quad
\alpha = \sqrt{3 - a_{10}^2 - 2a_{10}}.
\]
Since \( 1 < p(u_1) = 1 + a_{10} < 2 \), we have
\[
3 - a_{10}^2 - 2a_{10}
= 4 - \big(p(u_1)\big)^2 > 0,
\]
which ensures that the eigenvalues form a complex-conjugate pair lying on the unit circle, i.e., \( |\lambda_{1,2}| = 1 \).

The corresponding complex eigenvectors are:
\[
v_{1,2} =
\begin{bmatrix}
-\dfrac{u_1^2}{2(\gamma^2+u_1^2)} \\[6pt]
1
\end{bmatrix}
\mp i
\begin{bmatrix}
\dfrac{\alpha u_1^2}{2(\gamma^2 - u_1^2 + 2u_1^3)} \\[6pt]
0
\end{bmatrix}.
\]

\textbf{Step 4.} To compute the normal form of system~\eqref{bif2}, we rewrite system~\eqref{bif8} as:
\begin{equation}\label{sf}
\mathbf{x}^{(1)} = J \cdot \mathbf{x} + H(\mathbf{x}),
\end{equation}
where \( \mathbf{x} = (x, y)^{T} \), and \( H(\mathbf{x}) \) is the nonlinear part of system~\eqref{bif8} excluding the \( O(\cdot) \) terms, given by:
\[
H(\mathbf{x}) =
\begin{bmatrix}
a_{20}x^2 + a_{11}xy + a_{30}x^3 + a_{21}x^2y \\[6pt]
b_{20}x^2 + b_{11}xy + b_{30}x^3 + b_{21}x^2y
\end{bmatrix}.
\]

Define the transformation matrix:
\[
T =
\begin{bmatrix}
\dfrac{\alpha u_1^2}{2(\gamma^2 - u_1^2 + 2u_1^3)} & -\dfrac{u_1^2}{2(\gamma^2+u_1^2)} \\[8pt]
0 & 1
\end{bmatrix}
=
\begin{bmatrix}
mn & -n \\[4pt]
0 & 1
\end{bmatrix},
\]
where
\[
m = \frac{\gamma^2+u_1^2}{\gamma^2 - u_1^2 + 2u_1^3},
\quad
n = \frac{u_1^2}{2(\gamma^2+u_1^2)}.
\]
Then
\[
T^{-1} =
\begin{bmatrix}
\dfrac{1}{mn} & \dfrac{1}{m} \\[4pt]
0 & 1
\end{bmatrix}.
\]

Applying the transformation
\[
\begin{bmatrix}
x \\ y
\end{bmatrix}
= T \cdot
\begin{bmatrix}
X \\ Y
\end{bmatrix},
\]
system~\eqref{bif8} becomes:
\begin{equation}\label{sf1}
\mathbf{X}^{(1)} = T^{-1} J T \, \mathbf{X} + T^{-1} H(T\mathbf{X}) + O(\rho_1),
\end{equation}
where \( \mathbf{X} = (X, Y)^{T} \), and \( \rho_1 = \sqrt{X^2 + Y^2} \).

Let:
\[
H(T\mathbf{X}) = \begin{bmatrix} f(X, Y) \\ g(X, Y) \end{bmatrix},
\]
where
\begin{align*}
f(X, Y) &= a_{20}m^2n^2X^2 + (a_{11}mn - a_{20}mn^2)XY + (a_{20}n^2 - a_{11}n)Y^2 \\
&\quad + a_{30}m^3n^3X^3 + (a_{21}m^2n^2 - 3a_{30}m^2n^3)X^2Y \\
&\quad + (3a_{30}mn^3 - 2a_{21}mn^2)XY^2 + (a_{21}n^2 - a_{30}n^3)Y^3,
\end{align*}
\begin{align*}
g(X, Y) &= b_{20}m^2n^2X^2 + (b_{11}mn - b_{20}mn^2)XY + (b_{20}n^2 - b_{11}n)Y^2 \\
&\quad + b_{30}m^3n^3X^3 + (b_{21}m^2n^2 - 3b_{30}m^2n^3)X^2Y \\
&\quad + (3b_{30}mn^3 - 2b_{21}mn^2)XY^2 + (b_{21}n^2 - b_{30}n^3)Y^3.
\end{align*}

Now define:
\[
T^{-1} H(T\mathbf{X}) = \begin{bmatrix} F(X, Y) \\ G(X, Y) \end{bmatrix},
\]
where
\begin{align*}
F(X, Y) &= c_{20}X^2 + c_{11}XY + c_{02}Y^2 + c_{30}X^3 + c_{21}X^2Y + c_{12}XY^2 + c_{03}Y^3, \\
G(X, Y) &= d_{20}X^2 + d_{11}XY + d_{02}Y^2 + d_{30}X^3 + d_{21}X^2Y + d_{12}XY^2 + d_{03}Y^3,
\end{align*}
with coefficients:
\begin{align*}
&c_{20} = a_{20}mn + b_{20}mn^2, \quad
c_{11} = a_{11} - a_{20}n + b_{11}n - b_{20}n^2, \\
&c_{02} = \frac{a_{20}n - a_{11} + b_{20}n^2 - b_{11}n}{m}, \quad
c_{30} = a_{30}m^2n^2 + b_{30}m^2n^3, \\
&c_{21} = a_{21}mn - 3a_{30}mn^2 + b_{21}mn^2 - 3b_{30}mn^3, \\
&c_{12} = 3a_{30}n^2 - 2a_{21}n + 3b_{30}n^3 - 2b_{21}n^2, \\
&c_{03} = \frac{a_{21}n - a_{30}n^2 + b_{21}n^2 - b_{30}n^3}{m}, \\
&d_{20} = b_{20}m^2n^2, \quad
d_{11} = b_{11}mn - b_{20}mn^2, \\
&d_{02} = b_{20}n^2 - b_{11}n, \quad
d_{30} = b_{30}m^3n^3, \\
&d_{21} = b_{21}m^2n^2 - 3b_{30}m^2n^3, \\
&d_{12} = 3b_{30}mn^3 - 2b_{21}mn^2, \quad
d_{03} = b_{21}n^2 - b_{30}n^3.
\end{align*}

The relevant partial derivatives at the origin \( (0,0) \) are:
\begin{align*}
&F_{XX} = 2c_{20}, \quad F_{XY} = c_{11}, \quad F_{YY} = 2c_{02}, \\
&F_{XXX} = 6c_{30}, \quad F_{XXY} = 2c_{21}, \quad F_{XYY} = 2c_{12}, \quad F_{YYY} = 6c_{03}, \\
&G_{XX} = 2d_{20}, \quad G_{XY} = d_{11}, \quad G_{YY} = 2d_{02}, \\
&G_{XXX} = 6d_{30}, \quad G_{XXY} = 2d_{21}, \quad G_{XYY} = 2d_{12}, \quad G_{YYY} = 6d_{03}.
\end{align*}

\textbf{Step 5.} For the Neimark–Sacker bifurcation analysis, the discriminant quantity \( \mathcal{L} \) determining the stability of the emerging invariant closed curve is:
\begin{equation}\label{lya}
\mathcal{L} = -\operatorname{Re}\!\left[\frac{(1 - 2\lambda_1)\lambda_2^2}{1 - \lambda_1} L_{11} L_{20} \right] - \frac{1}{2} |L_{11}|^2 - |L_{02}|^2 + \operatorname{Re}(\lambda_2 L_{21}),
\end{equation}
where:
\begin{align*}
L_{20} &= \frac{1}{8} \left[(F_{XX} - F_{YY} + 2G_{XY}) + i(G_{XX} - G_{YY} - 2F_{XY}) \right], \\
L_{11} &= \frac{1}{4} \left[(F_{XX} + F_{YY}) + i(G_{XX} + G_{YY}) \right], \\
L_{02} &= \frac{1}{8} \left[(F_{XX} - F_{YY} - 2G_{XY}) + i(G_{XX} - G_{YY} + 2F_{XY}) \right], \\
L_{21} &= \frac{1}{16} \left[(F_{XXX} + F_{XYY} + G_{XXY} + G_{YYY}) \right. \\
&\quad \left. + i(G_{XXX} + G_{XYY} - F_{XXY} - F_{YYY}) \right].
\end{align*}

Based on this, we state the following result.

\begin{thm}\label{bifurcation}
Let \( \beta = \beta_0 \) satisfy \( q(u_1) = 1 \). If \( \beta \) varies in a sufficiently small neighborhood of \( \beta_0 \), the system~\eqref{h12} undergoes a Neimark-Sacker bifurcation at \( E_1 = (u_1, v_1) \). Moreover, if \( \mathcal{L} < 0 \) (resp. \( \mathcal{L} > 0 \)), an attracting (resp. repelling) invariant closed curve bifurcates from \( E_1 \) for \( \beta > \beta_0 \) (resp. \( \beta < \beta_0 \)).
\end{thm}

\section{Global dynamics of (\ref{h12})}

It can be verified that the set
\[
\mathcal{U} = \{(u,v) \in \mathbb{R}_+^2 \mid 0 \leq u \leq 2,\ v = 0\}
\]
is invariant with respect to the operator~\eqref{h12}.

Moreover, if \( 0 < r \leq 1 \), then the set
\[
\mathcal{V} = \{(u,v) \in \mathbb{R}_+^2 \mid u = 0,\ v \geq 0\}
\]
is also invariant with respect to this operator.

In~\cite{SH-4}, it is shown that any trajectory starting from a point in \( \mathcal{U} \) converges to the fixed point \( (1,0) \),
while any trajectory originating in \( \mathcal{V} \) converges to the fixed point \( (0,0) \).

\begin{equation}\label{poly}
\psi(u)=(\beta-\theta+1-r)u^3+(\beta+1-r)\gamma u^2+(1-r-\theta)\gamma^2 u+(1-r)\gamma^3
\end{equation}

\begin{pro}\label{propv}
Let \( v^{(1)} \) be defined as in~\eqref{h12}.
If one of the following conditions on the parameters holds, then \( v^{(1)} \geq 0 \) for all \( u \in [0,1] \) and \( v \geq 0 \):

\begin{itemize}
    \item[(i)] \( 0 < \theta \leq 1 \)
    \begin{itemize}
        \item[(i.1)] \( 0 < r \leq 1 - \theta \), \( \beta > 0 \), \( \gamma > 0 \);
        \item[(i.2)] \( 1-\theta < r \leq \frac{4 - \theta}{4} \), \( 0<\beta<r+\theta-1 \), \( \gamma > \gamma_1 \);
        \item[(i.3)] \( 1 - \theta < r \leq 1 - \theta(3 - 2\sqrt{2}) \), \( \beta \geq r + \theta - 1 \), \( \gamma > 0 \);
        \item[(i.4)] \( \frac{4 - \theta}{4} < r < 1 \), \( 0 < \beta < \beta_2 \), \( \gamma \geq \gamma_1 \);
        \item[(i.5)] \( \frac{4 - \theta}{4} < r \leq 1 - \theta(3 - 2\sqrt{2}) \), \( \beta_2 \leq \beta < \beta_3 \), \( \gamma \geq \gamma_3 \);
        \item[(i.6)] \( \frac{4 - \theta}{4} < r < 1 - \theta(3 - 2\sqrt{2}) \), \( \beta_3 \leq \beta < r + \theta - 1 \), \( \gamma \geq \gamma_1 \);
        \item[(i.7)] \( 1 - \theta(3 - 2\sqrt{2}) < r < 1 \), \( \beta_2 \leq \beta < \beta_3 \), \( \gamma \geq \gamma_3 \);
        \item[(i.8)] \( 1 - \theta(3 - 2\sqrt{2}) < r < 1 \), \( \beta \geq \beta_3 \), \( \gamma > 0 \);
    \end{itemize}

    \item[(ii)] \( 1 < \theta \leq 4 \)
    \begin{itemize}
        \item[(ii.1)] \( 0 < r \leq \frac{4 - \theta}{4} \), \( 0 < \beta < r + \theta - 1 \), \( \gamma \geq \gamma_1 \);
        \item[(ii.2)] \( 0 < r \leq 1 - \theta(3 - 2\sqrt{2}) \), \( \beta \geq r + \theta - 1 \), \( \gamma > 0 \);
        \item[(ii.3)] \( \frac{4 - \theta}{4} < r < 1 \), \( 0 < \beta < \beta_2 \), \( \gamma \geq \gamma_1 \);
        \item[(ii.4)] \( \frac{4 - \theta}{4} < r \leq 1 - \theta(3 - 2\sqrt{2}) \), \( \beta_2 \leq \beta < \beta_3 \), \( \gamma \geq \gamma_3 \);
        \item[(ii.5)] \( \frac{4 - \theta}{4} < r < 1 - \theta(3 - 2\sqrt{2}) \), \( \beta_3 \leq \beta < r + \theta - 1 \), \( \gamma \geq \gamma_1 \);
        \item[(ii.6)] \( 1 - \theta(3 - 2\sqrt{2}) < r < 1 \), \( \beta_2 \leq \beta < \beta_3 \), \( \gamma \geq \gamma_3 \);
        \item[(ii.7)] \( 1 - \theta(3 - 2\sqrt{2}) < r < 1 \), \( \beta \geq \beta_3 \), \( \gamma > 0 \);
    \end{itemize}

    \item[(iii)] \( 4 < \theta < 3 + 2\sqrt{2} \)
    \begin{itemize}
        \item[(iii.1)] \( 0 < r < 1 \), \( 0 < \beta < \beta_2 \), \( \gamma \geq \gamma_1 \);
        \item[(iii.2)] \( 0 < r \leq 1 - \theta(3 - 2\sqrt{2}) \), \( \beta_2 \leq \beta < \beta_3 \), \( \gamma \geq \gamma_3 \);
        \item[(iii.3)] \( 0 < r < 1 - \theta(3 - 2\sqrt{2}) \), \( \beta_3 \leq \beta < r + \theta - 1 \), \( \gamma \geq \gamma_1 \);
        \item[(iii.4)] \( 0 < r \leq 1 - \theta(3 - 2\sqrt{2}) \), \( \beta \geq r + \theta - 1 \), \( \gamma > 0 \);
        \item[(iii.5)] \( 1 - \theta(3 - 2\sqrt{2}) < r < 1 \), \( \beta_2 \leq \beta < \beta_3 \), \( \gamma \geq \gamma_3 \);
        \item[(iii.6)] \( 1 - \theta(3 - 2\sqrt{2}) < r < 1 \), \( \beta \geq \beta_3 \), \( \gamma > 0 \);
    \end{itemize}

    \item[(iv)] \( \theta \geq 3 + 2\sqrt{2} \)
    \begin{itemize}
        \item[(iv.1)] \( 0 < r < 1 \), \( 0 < \beta < \beta_2 \), \( \gamma \geq \gamma_1 \);
        \item[(iv.2)] \( 0 < r < 1 \), \( \beta_2 \leq \beta < \beta_3 \), \( \gamma \geq \gamma_3 \);
        \item[(iv.3)] \( 0 < r < 1 \), \( \beta \geq \beta_3 \), \( \gamma > 0 \).
    \end{itemize}
\end{itemize}
where \( \gamma_1 \) and \( \gamma_3 \) ($\gamma_1<\gamma_3$) are the solutions of the following cubic equation:
\begin{equation}\label{gammaeq}
(1 - r)\gamma^3 + (1 - r - \theta)\gamma^2 + (\beta + 1 - r)\gamma + \beta + 1 - r - \theta = 0,
\end{equation}
and \( \beta_2 \) and \( \beta_3 \) ($\beta_2<\beta_3$) are the solutions of the cubic equation:
\begin{equation}\label{betaeq}
\begin{aligned}
&(4r - 4)\beta^3 + (\theta^2 - 20\theta + 20r\theta + 40r - 20r^2 - 20)\beta^2 \\
&\quad + (4\theta^3 + 8\theta^2 - 8r\theta^2 + 8\theta - 16r\theta + 8r^2\theta + 32r^3 - 96r^2 + 96r - 32)\beta \\
&\quad - 4\theta^4 + 16\theta^3 - 16r\theta^3 - 32\theta^2 + 64r\theta^2 - 32r^2\theta^2 + 32\theta - 96r\theta \\
&\quad + 96r^2\theta - 32r^3\theta - 16r^4 + 64r^3 - 96r^2 + 64r - 16 = 0.
\end{aligned}
\end{equation}

\end{pro}

\begin{proof}
Consider the condition \( v^{(1)} \geq 0 \), which is equivalent to
\[
\frac{v\left((\beta+1-r-\theta)u^3+(\beta+1-r)\gamma u^2+(1-r-\theta)\gamma^2u+(1-r)\gamma^3\right)}{(\gamma^2+u^2)(\gamma+u)} \geq 0.
\]

Introduce the auxiliary function:
\[
\mathcal{K}(u) = (\beta + 1 - r - \theta)u^3 + (\beta + 1 - r)\gamma u^2 + (1 - r - \theta)\gamma^2 u + (1 - r)\gamma^3,
\]
so that \( \mathcal{K}(0) = (1 - r)\gamma^3 \), and the value at \( u = 1 \), which coincides with the left-hand side of the cubic equation~\eqref{gammaeq}, is:
\[
\mathcal{K}(1) = (1 - r)\gamma^3 + (1 - r - \theta)\gamma^2 + (\beta + 1 - r)\gamma + \beta + 1 - r - \theta.
\]

Under the parameter conditions of case (i.1), it is evident that \( \mathcal{K}(u) \geq 0 \) for all \( u \in (0,1) \).

Since \( \mathcal{K}(u) \) is a cubic function, it can have at most two local extrema, occurring at points \( \widetilde{u}_1 \) and \( \widetilde{u}_2 \), where:
\[
\widetilde{u}_{1} = \frac{\gamma\left[\beta + 1 - r - \sqrt{
(\beta + 1 - r)^2 - 3(1 - r - \theta)(\beta + 1 - r - \theta)}\right]}{3(\beta + 1 - r - \theta)},
\]
\[
\widetilde{u}_{2} = \frac{\gamma\left[\beta + 1 - r + \sqrt{
(\beta + 1 - r)^2 - 3(1 - r - \theta)(\beta + 1 - r - \theta)}\right]}{3(\beta + 1 - r - \theta)}.
\]

In order for \( \mathcal{K}(u) \geq 0 \) to hold for all \( u \in (0,1) \), the condition \( 0 < r < 1 \) is necessary. Now consider the sign of the leading coefficient of \( \mathcal{K}(u) \), which is \( \beta + 1 - r - \theta \):

\begin{itemize}
    \item If \( \beta + 1 - r - \theta < 0 \), then it suffices to verify the inequality \( \mathcal{K}(1) \geq 0 \).
    \item If \( \beta + 1 - r - \theta \geq 0 \), then the following two scenarios must be considered:
    \begin{enumerate}
        \item \( \mathcal{K}_{\min} = \mathcal{K}(\widetilde{u}_2) \geq 0 \);
        \item \( \mathcal{K}(\widetilde{u}_2) < 0 \), \( \widetilde{u}_2 \geq 1 \), and \( \mathcal{K}(1) \geq 0 \).
    \end{enumerate}
\end{itemize}

Evaluating \( \mathcal{K}(\widetilde{u}_2) \) leads to the cubic equation~\eqref{betaeq}. By analyzing all possible parameter cases and generalizing the results, the proof of the proposition is complete.
\end{proof}

Define the function
\begin{equation}\label{omega}
\omega(u) = \frac{(2 - u)(\gamma^2 + u^2)}{u},
\end{equation}

\begin{pro}\label{lim1}
Assume that the parameters satisfy one of the conditions given in Proposition~\ref{propv}, and that $\beta \leq \Psi(1)$ and $\gamma \geq \sqrt{2} - 1$. Then, for any initial point $(u,v) \in M$ with $u > 0$ and $v \neq \omega(u)$, the iterates of the operator \eqref{h12} converge to the fixed point $(1,0)$, where
\[
M = \left\{ (u,v) \in \mathbb{R}^2 \,\middle|\, 0 \leq u \leq 1,\; 0 \leq v \leq \omega(u) \right\}.
\]
\end{pro}

\begin{proof}
For all $1 < u \leq 2$, we have
\[
u^{(1)} = u(2 - u) - \frac{u^2v}{\gamma^2 + u^2} \leq u(2 - u) \leq 1,
\]
so it suffices to consider the case $0 < u \leq 1$. Let $(u,v) \in M$. Since $v < \omega(u)$, it follows that $0 < u^{(1)} \leq 1$.

Consider the derivative of the function (\ref{omega}):
\[
\omega'(u) = \frac{2(-u^3 + u^2 - \gamma^2)}{u^2}.
\]
The function
\begin{equation}\label{ro}
\rho(u)=-u^3 + u^2 - \gamma^2
\end{equation}
attains its maximum value $\frac{4 - 27\gamma^2}{27}$ at $u = 2/3$. Thus, $\omega(u)$ is decreasing for $\gamma \geq \frac{2}{3\sqrt{3}}$. Under the assumption $\gamma \geq \sqrt{2} - 1 > \frac{2}{3\sqrt{3}}$, the function $\omega(u)$ is strictly decreasing.

From Proposition~\ref{propv}, we know that $v^{(1)} \geq 0$. Moreover, Theorem~\ref{thm1} implies that there is no positive fixed point when $\beta \leq \Psi(1)$ and $\gamma \geq \sqrt{2} - 1$. Therefore, $\beta \leq \Psi(u)$ for all $u \in [0,1]$, which implies that the sequence $\{v^{(n)}\}$ is decreasing.

We now show that $v^{(1)} < \omega(u^{(1)})$. Note that $u^{(1)} \leq u$ when
\[
v \geq \frac{(1 - u)(\gamma^2 + u^2)}{u},
\]
and $u^{(1)} \geq u$ when
\[
v \leq \frac{(1 - u)(\gamma^2 + u^2)}{u}.
\]
In the case $v \leq \frac{(1 - u)(\gamma^2 + u^2)}{u}$, since $\{v^{(n)}\}$ is decreasing and $u^{(1)} \geq u$, the inequality $v^{(1)} < \omega(u^{(1)})$ holds trivially.

Now assume $v \geq \frac{(1 - u)(\gamma^2 + u^2)}{u}$, i.e., $u^{(1)} \leq u$. Since $\omega(u)$ is decreasing, and $\{v^{(n)}\}$ is decreasing, we have
\[
v^{(1)} \leq v < \omega(u) < \omega(u^{(1)}),
\]
which again confirms $v^{(1)} < \omega(u^{(1)})$. Therefore, the set $M$ is invariant under the operator~\eqref{h12}.

Now fix a horizontal line $v = \nu > 0$ and define the compact, positively invariant subset
\[
N = \left\{ (u,v) \in M \,\middle|\, 0 \leq v \leq \nu \right\} \subset M.
\]
Clearly, for any $(u,v) \in M$, we can choose $\nu$ such that $(u,v) \in N$.

Define the Lyapunov function \( V : N \to \mathbb{R} \) by
\[
V(u,v) = v.
\]
It satisfies \( V(u,v) \geq 0 \) for all \( (u,v) \in N \) and \( V(1,0) = 0 \). The difference along the trajectory is
\[
\Delta V = V(u^{(1)}, v^{(1)}) - V(u,v) = v^{(1)} - v \leq 0,
\]
since we have already shown that $v^{(1)} \leq v$. Thus, \( V \) is non-increasing along trajectories.

The set where \( \Delta V = 0 \) is
\[
\mathcal{B} = \left\{ (u,v) \in N \;\middle|\; \Delta V = 0 \right\} = \left\{ (u, 0) \right\}.
\]
For \( u > 0 \), the only invariant point in \( \mathcal{B} \) is the fixed point \( (1, 0) \).

Since \( N \) is compact and positively invariant, and \( V(u,v) \) is continuous and non-increasing along trajectories, LaSalle's Invariance Principle implies that all trajectories starting with \( 0 < u \leq 1 \) and \( 0 < v < \omega(u) \) converge to the largest invariant subset of \( \mathcal{B} \), namely the singleton \( \{(1,0)\} \).
This completes the proof.
\end{proof}

\begin{pro}\label{lim2}
Assume that the parameters satisfy one of the conditions given in Proposition~\ref{propv}, and that
\[
\frac{2}{3\sqrt{3}} \leq \gamma < \sqrt{2} - 1.
\]
If either of the following conditions holds:
\begin{enumerate}
    \item[(i)] $0 < \theta \leq \dfrac{2r\gamma(1+\gamma)^2}{1 - 2\gamma - \gamma^2}$ and $\beta \leq \Psi(1)$;
    \item[(ii)] $\theta > \dfrac{2r\gamma(1+\gamma)^2}{1 - 2\gamma - \gamma^2}$ and $\beta < \Psi(\widehat{u})$,
\end{enumerate}
where $\widehat{u}$ is the unique local minimizer of $\Psi(u)$ in the interval $(0,1)$, then for any initial point $(u,v) \in M$ with $u > 0$ and $v \neq \omega(u)$, the iterates of the operator~\eqref{h12} converge to the fixed point $(1,0)$.
\end{pro}

\begin{proof}
Under the given conditions on the parameters, there exists no positive fixed point (by Theorem~\ref{thm1}), and the function $\omega(u)$ defined in (\ref{omega}), is strictly decreasing on the interval $(0,1]$. The remainder of the proof follows the same arguments as in the proof of Proposition~\ref{lim1}, relying on the invariance of the set $M$, the monotonicity of the sequence $\{v^{(n)}\}$, and the use of a Lyapunov function to apply LaSalle's Invariance Principle. Hence, all trajectories with $u > 0$ and
$v \neq \omega(u)$ converge to the fixed point $(1,0)$.
\end{proof}

Recall that the function $\rho(u)$, defined in~\eqref{ro}, is a cubic polynomial that takes the value $-\gamma^2$ at both endpoints $u = 0$ and $u = 1$. Moreover, for $0 < \gamma < \frac{2}{3\sqrt{3}}$, the function $\rho(u)$ attains a positive maximum, i.e., $\rho_{\max} > 0$, which implies that the function $\omega(u)$ has two critical points $\overline{x}_1$ and $\overline{x}_2$ in the interval $(0,1)$, with $\overline{x}_1 < \overline{x}_2$.

Consequently, $\omega(u)$ is decreasing on the intervals $(0, \overline{x}_1) \cup (\overline{x}_2, 1)$ and increasing on $(\overline{x}_1, \overline{x}_2)$. Therefore, $\omega(\overline{x}_1)$ is the global minimum and $\omega(\overline{x}_2)$ is the global maximum of $\omega(u)$ on $(0,1)$.

We now arrive at the following proposition.

\begin{pro}\label{lim3}
Assume that the parameters satisfy one of the conditions given in Proposition~\ref{propv}, and that
\[
0 < \gamma < \frac{2}{3\sqrt{3}}.
\]
Let the initial point $(u,v)$ belong to the set $M$, and suppose one of the following conditions holds:
\begin{enumerate}
    \item[(i)] $0 < u \leq \overline{x}_1$ and $v \neq \omega(u)$;
    \item[(ii)] $u > 0$ and $v < \omega(\overline{x}_1)$.
\end{enumerate}
Then the iterates of the operator~\eqref{h12} converge to the fixed point $(1,0)$.
\end{pro}

Thus, the case in which there is no positive fixed point has been fully analyzed.

Now, assume that there exists a unique positive fixed point $E_1 = (u_1, v_1)$.
According to Theorem~\ref{thm1}, a necessary condition for this is
\[
\beta > \Psi(1) = \frac{(r + \theta + r\gamma)(\gamma^2 + 1)}{\gamma + 1}.
\]
Under this condition, and by Proposition~\ref{prop1}, the fixed point $(1,0)$ is a saddle (unstable).
Numerical analysis shows that if the unique positive fixed point $E_1 = (u_1, v_1)$ is attracting (i.e., $q(u_1) < 1$) and the parameters satisfy one of the conditions given in Proposition~\ref{propv}, then any trajectory starting from a positive initial point lying below the curve $\omega(u)$ converges to $E_1$.

If there exist two positive fixed points $E_1 = (u_1, v_1)$ and $E_2 = (u_2, v_2)$, then the fixed point $E_1$ cannot be globally attractive. Since the fixed point $E_2$ is always a saddle, from general theory there exists an invariant curve passing through the fixed point $E_2$, and the set $M$ is divided into two parts: $M_1$ and $M_2$, where $M_1$ contains the fixed point $E_1$ while $M_2$ contains the fixed point $(1,0)$. In this case, two possibilities are observed:
\begin{itemize}
    \item  If the fixed point $E_1 = (u_1, v_1)$ is attracting, then any trajectory starting from $M_1$ converges to the fixed point $E_1$, while any trajectory starting from $M_2$ converges to $(1,0)$.
    \item  If the fixed point $E_1$ is repelling and no invariant closed curve about it is formed, then any trajectory starting from $M$ converges to the fixed point $(1,0)$.
\end{itemize}

\section{numerical simulations}

\textbf{Example 1.} (\emph{Case of a Unique Positive Fixed Point})\\
Consider system~\eqref{h12} with parameters \( r = 0.5 \), \( \theta = 4 \), and \( \gamma = 1 \). Solving the equations \( \beta = \Psi(u) \) and \( q(u) = 1 \), we find the bifurcation value of \( \beta \) to be
\[
\beta_0 \approx 7.4838.
\]
Moreover, since \( \Psi(1) = 5 \) and \( \beta > \Psi(1) \), and the parameter triple \( (r, \gamma, \theta) \in R_2 \), it follows from case~(i) of Theorem~\ref{thm1} that there exists a unique positive fixed point. This fixed point is approximately
\[
E_1 \approx (0.6057,\, 0.8896).
\]

The associated eigenvalues (multipliers) of the Jacobian matrix at \( E_1 \) are
\[
\lambda_1 \approx 0.6058 - 0.7955i, \quad \lambda_2 \approx 0.6058 + 0.7955i,
\]
which form a complex conjugate pair.

The computed normal form coefficients are given by
\[
L_{20} \approx 0.0823 + 0.1340i, \quad L_{11} \approx -0.0412 - 0.2075i,
\]
\[
L_{02} \approx -0.1918 + 0.0907i, \quad L_{21} \approx 0.0095 - 0.0231i.
\]

The corresponding discriminant quantity is
\[
\mathcal{L} \approx -0.036383 < 0.
\]
Hence, by Theorem~\ref{bifurcation}, the system~\eqref{h12} undergoes a \textbf{Neimark--Sacker bifurcation}, and an \textbf{attracting} invariant closed curve emerges from the fixed point when \( \beta > \beta_0 \) (i.e., when \( \beta^* > 0 \)).

Figure~\ref{diag} illustrates the Neimark--Sacker bifurcation diagram of system~\eqref{h12} along with the corresponding Maximal Lyapunov Exponent. In Figure~\ref{fig1}(a), the fixed point remains attracting since \( q(u_1) < 1 \). In Figure~\ref{fig1}(b), an invariant closed curve is born. Figures~\ref{fig1}(c), (d), and (e) demonstrate the expansion and deformation of the invariant closed curves. Finally, Figures~\ref{fig1}(f), (g), and (h) display trajectories starting from the initial point \( (0.1, 0.3) \) for various values of \( \beta \).

\begin{figure}[h!]
    \centering
    \subfigure[]{\includegraphics[width=0.7\textwidth]{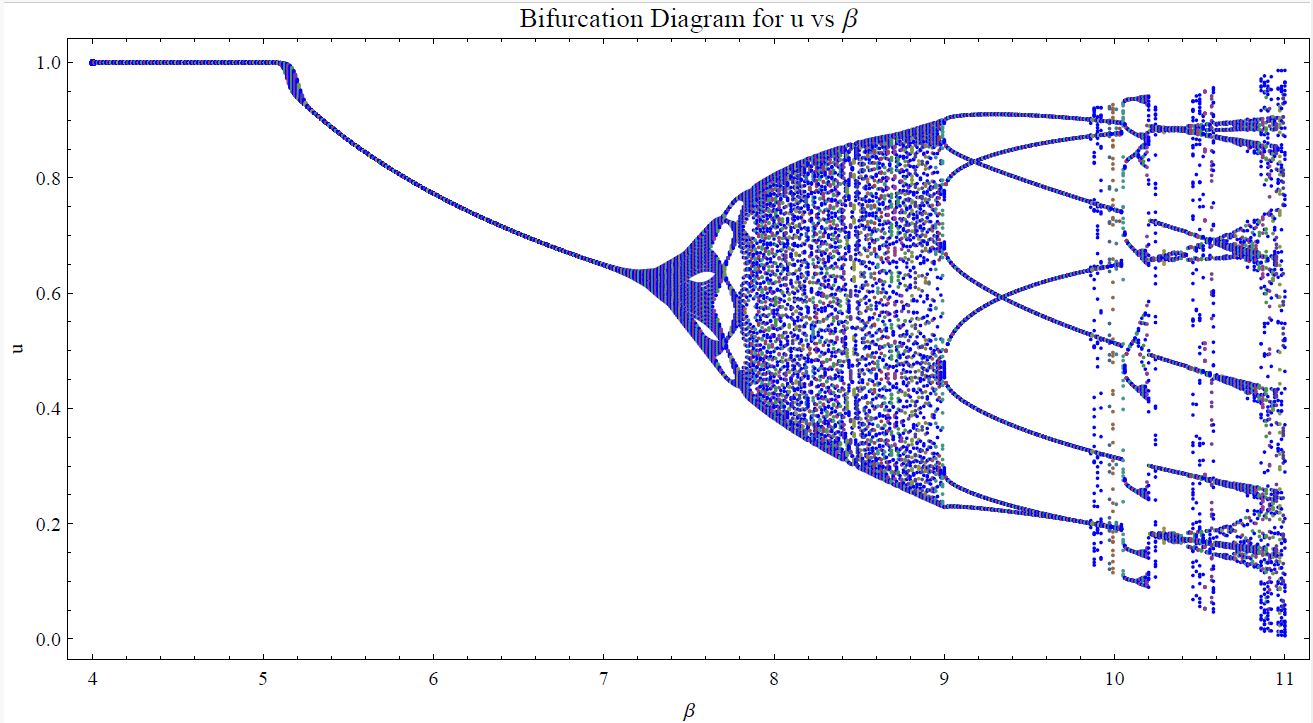}} \hspace{0.3in}
    \subfigure[]{\includegraphics[width=0.7\textwidth]{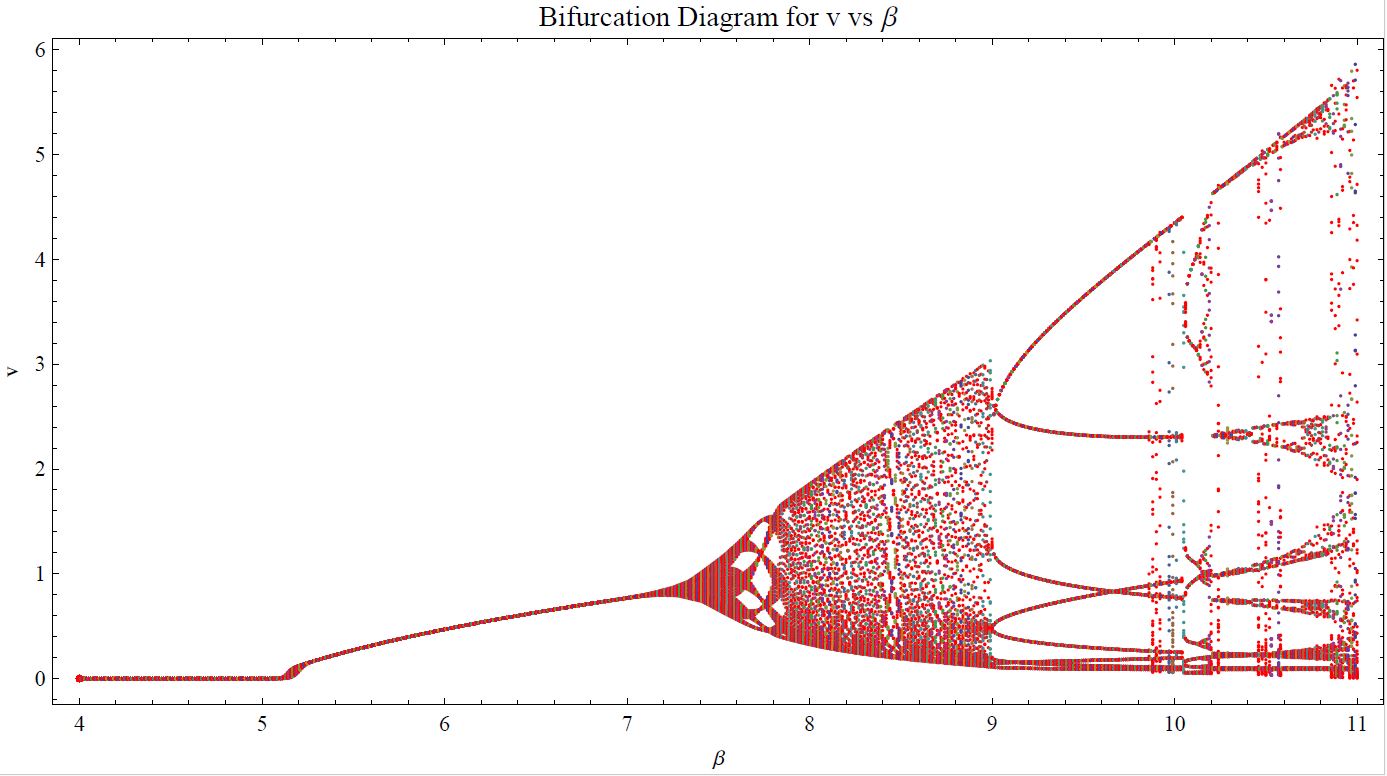}}
    \subfigure[]{\includegraphics[width=0.7\textwidth]{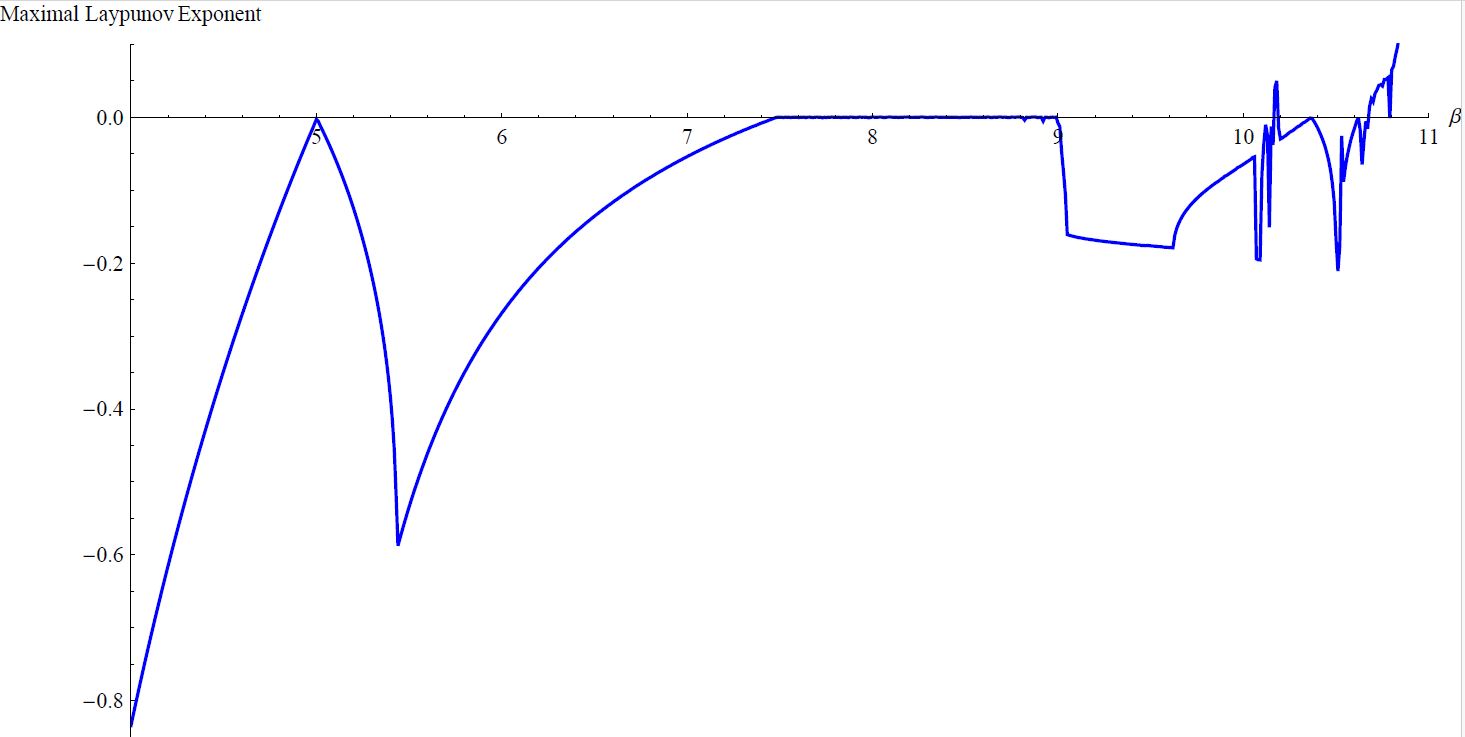}}
    \caption{Bifurcation diagrams for the system (\ref{h12}) with parameters \( r = 0.5 \), \( \gamma =1 \), \( \theta =4 \), and initial values \( u^0 = 0.1 \), \( v^0 = 0.3 \), as the bifurcation parameter \( \beta \) varies in the interval [4,11]. In (c), the maximal Lyapunov exponent corresponding to (a) and (b) is presented.}
    \label{diag}
\end{figure}

\begin{figure}[h!]
    \centering
    \subfigure[\tiny$\beta=7.46, u_1\approx0.6077, q(u_1)\approx0.9955, (0.1, 0.3).$]{\includegraphics[width=0.45\textwidth]{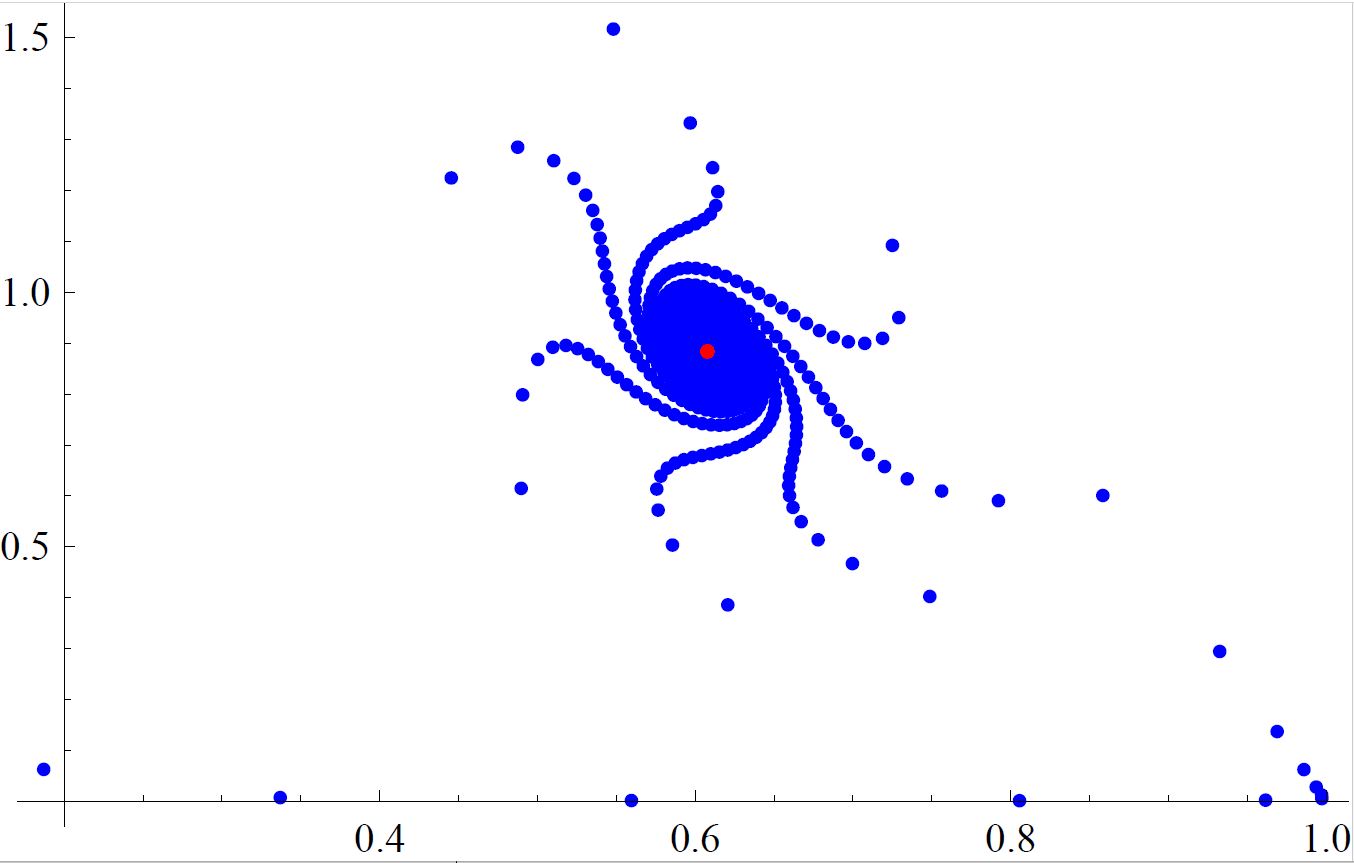}}
    \subfigure[\tiny$\beta=7.5, u_1\approx0.6044, q(u_1)\approx1.003, (0.1, 0.3).$]{\includegraphics[width=0.45\textwidth]{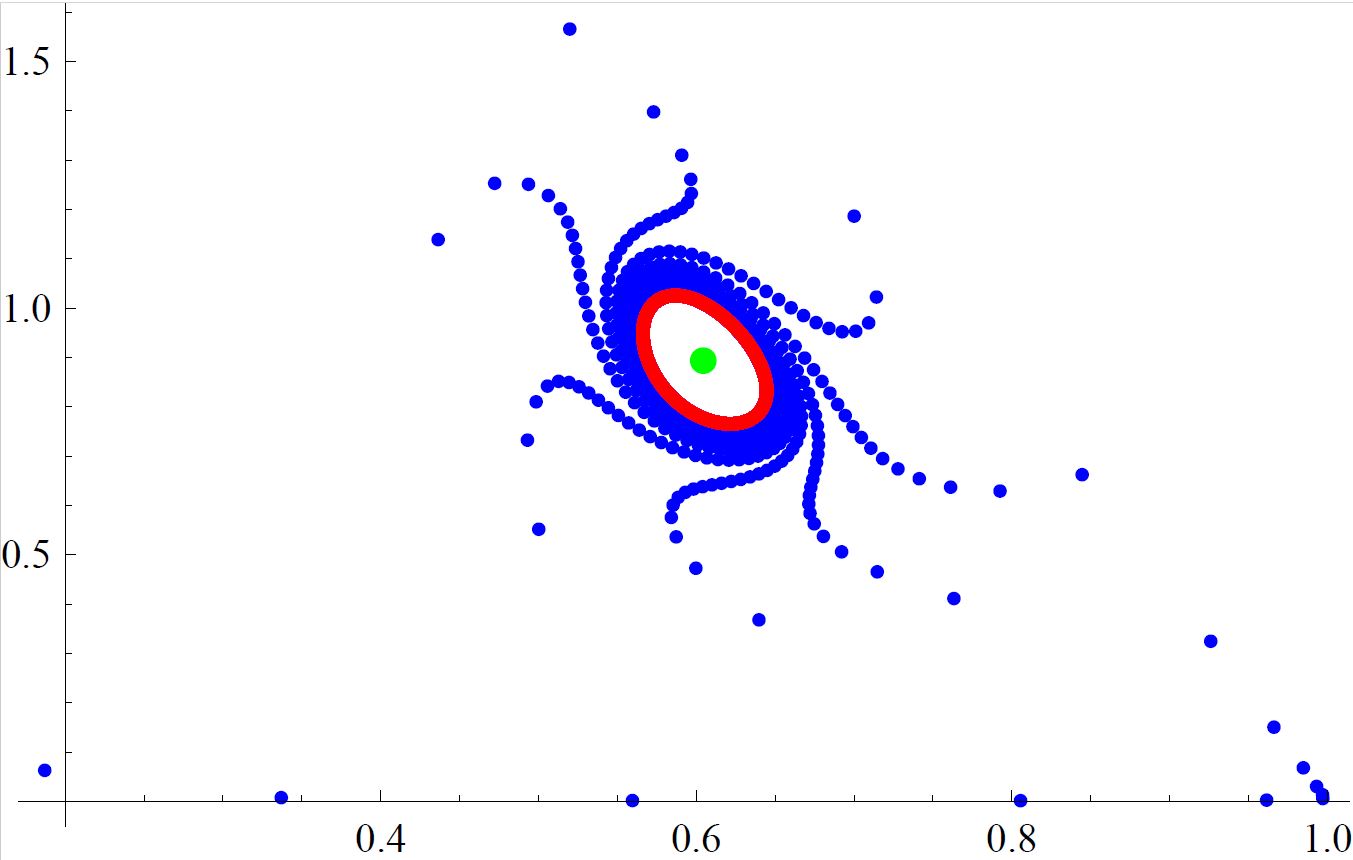}} \hspace{0.3in}
    \subfigure[\tiny$\beta=7.6, u_1\approx0.5965, q(u_1)\approx1.0212, (0.1, 0.3).$]{\includegraphics[width=0.44\textwidth]{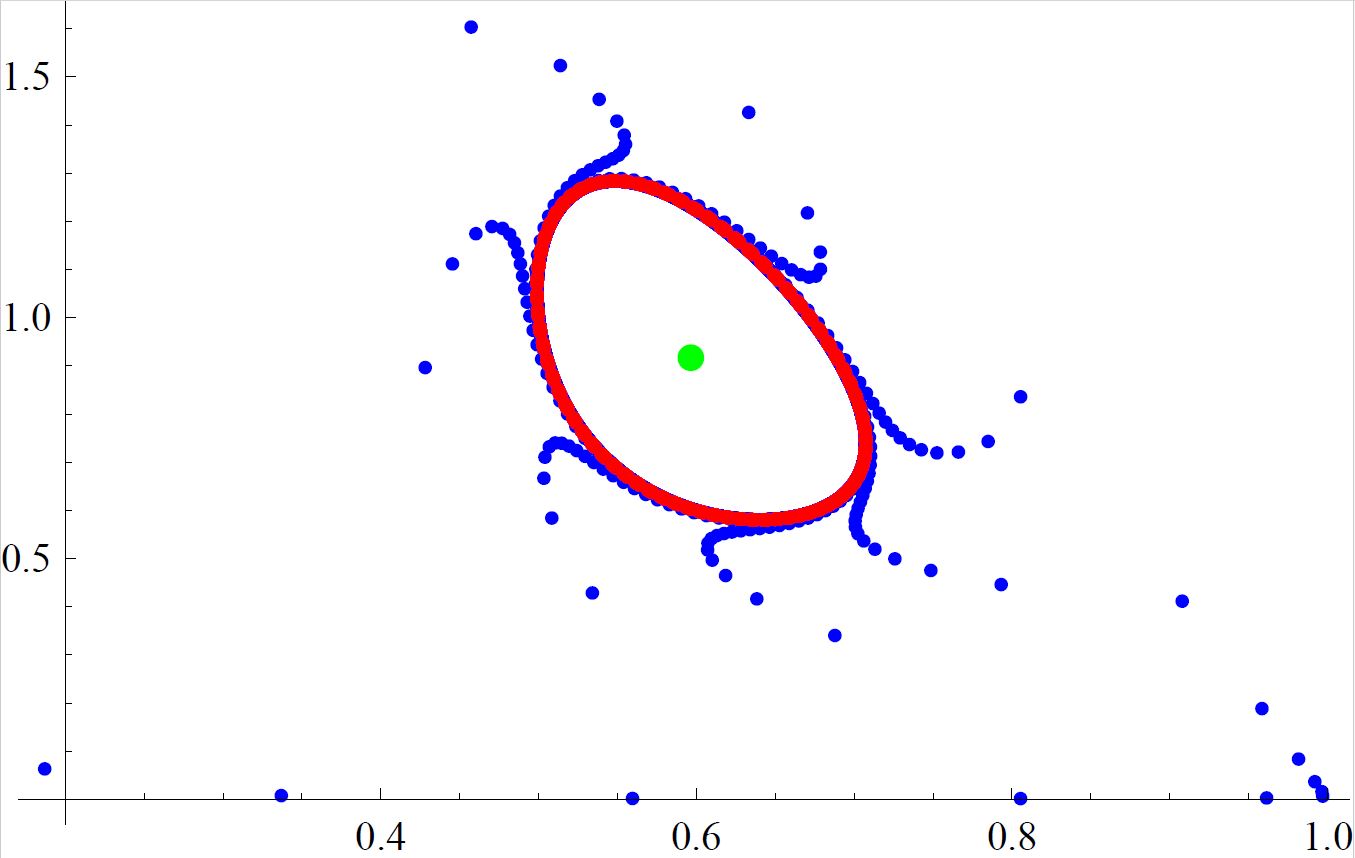}} \hspace{0.3in}
    \subfigure[\tiny$\beta=7.6, u_1\approx0.5965, q(u_1)\approx1.0212, (0.62, 0.84).$]{\includegraphics[width=0.44\textwidth]{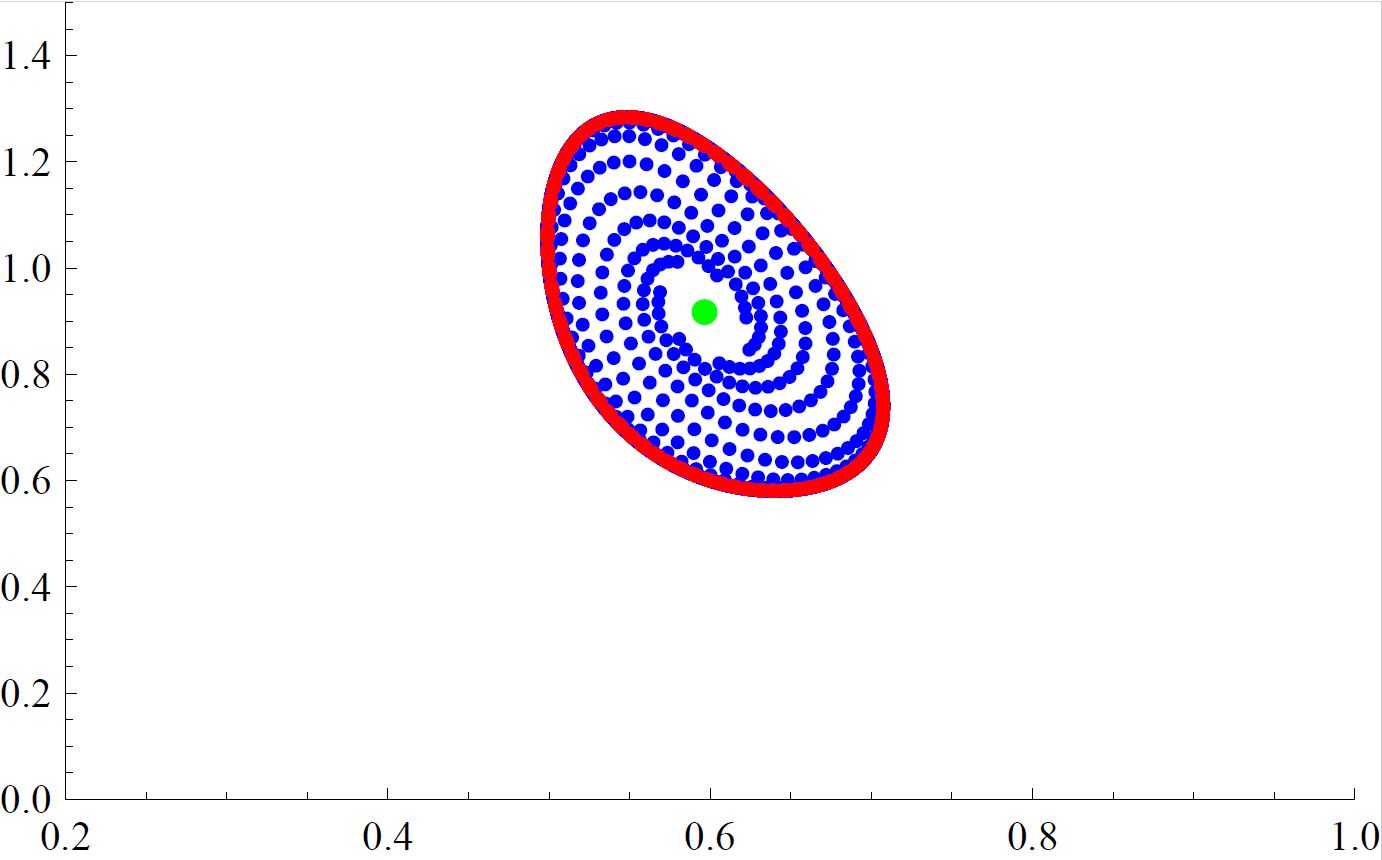}}\hspace{0.3in}
    \subfigure[\tiny$\beta=8.9, u_1\approx0.5135, q(u_1)\approx1.1967, (0.1, 0.3).$]{\includegraphics[width=0.44\textwidth]{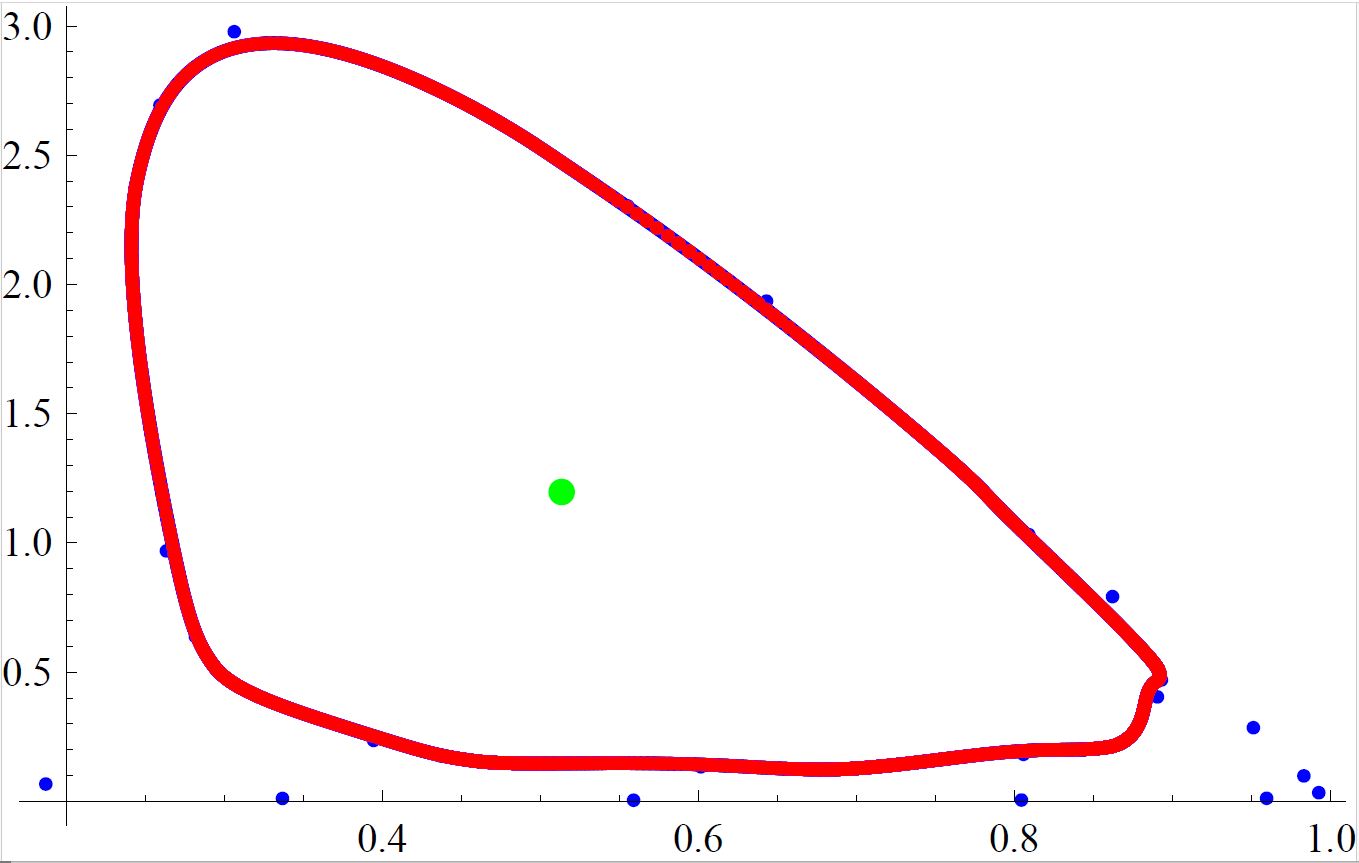}}\hspace{0.3in}
    \subfigure[\tiny$\beta=9.0, u_1\approx0.5083, q(u_1)\approx1.2127, (0.1, 0.3).$]{\includegraphics[width=0.44\textwidth]{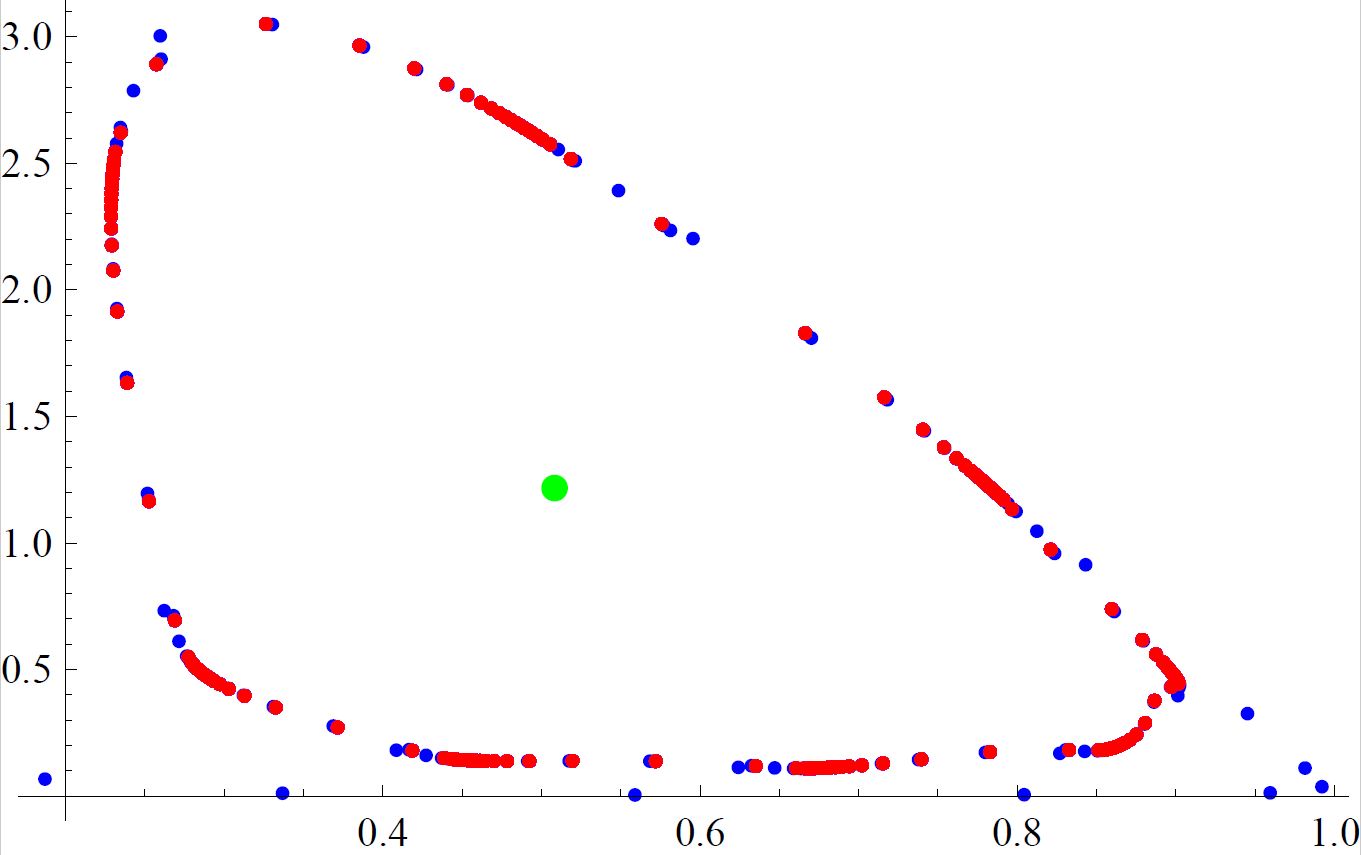}}\hspace{0.3in}
    \subfigure[\tiny$\beta=10.8, u_1\approx0.4337, q(u_1)\approx1.3313, (0.1, 0.3).$]{\includegraphics[width=0.44\textwidth]{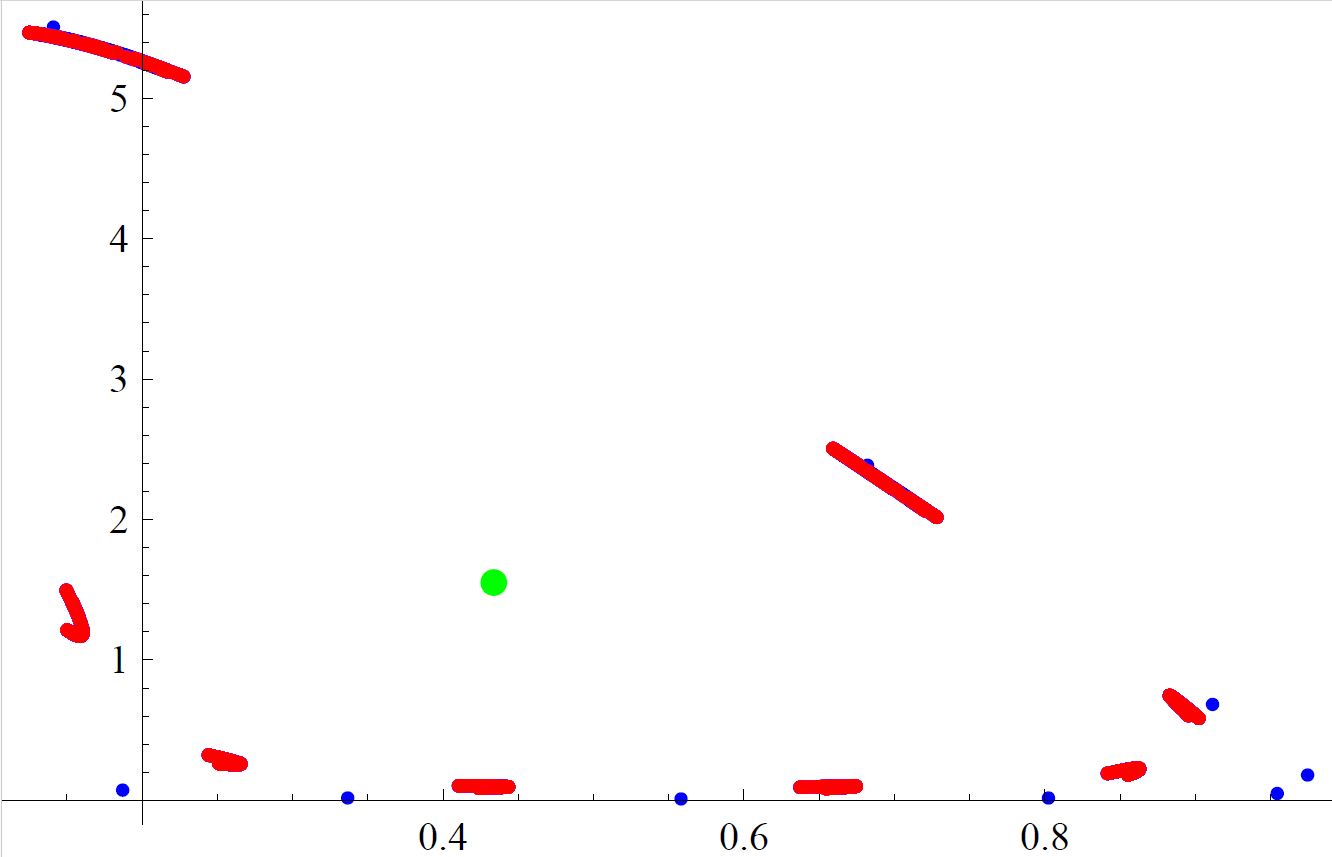}}\hspace{0.3in}
    \subfigure[\tiny$\beta=11, u_1\approx0.4271, q(u_1)\approx1.3406, (0.1, 0.3).$]{\includegraphics[width=0.44\textwidth]{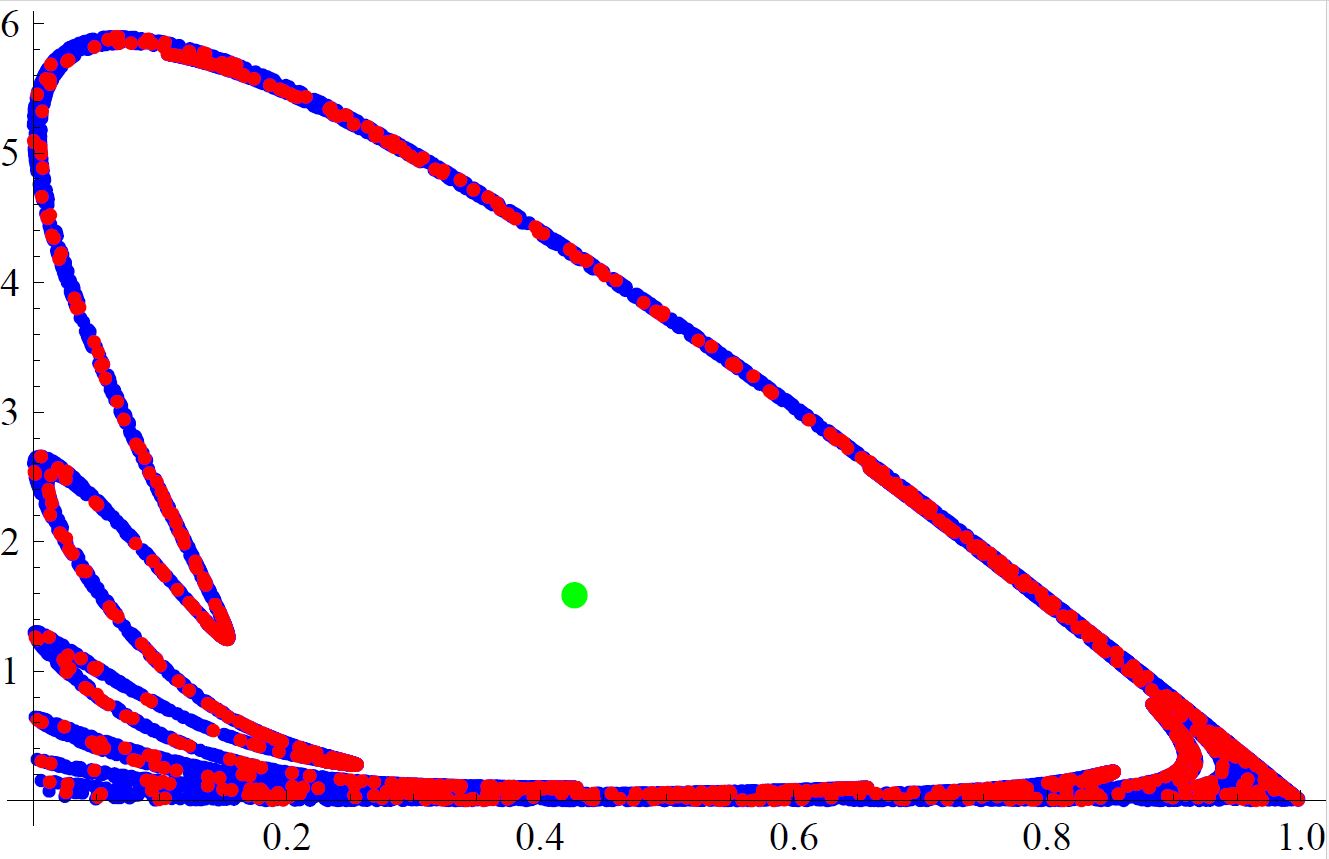}}
 \caption{Phase portraits of system~\eqref{h12} for parameters \( r = 0.5 \), \( \theta = 4 \), \( \gamma = 1 \), and \( n = 10{,}000 \). The green point represents the fixed point \( E_1 \), while the red points form a set of limit points. Panels~(e)--(h) display trajectories corresponding to various values of \( \beta \).}
\label{fig1}

\end{figure}

\textbf{Example 2.} (\emph{Case of Two Positive Fixed Points})\\
Consider system~\eqref{h12} with parameters \( r = 0.5 \), \( \theta = 5 \), and \( \gamma = 0.2 \). Then \( \Psi(1) = 4.85333 \), and the real solution of the equation \( h(u) = 0 \) is \( \widehat{u} = 0.5476 \). Evaluating the function \( \Psi \) at this point gives \( \Psi(\widehat{u}) = 4.71762 \).

According to Theorem~\ref{thm1}, if \( 4.71762 < \beta < 4.85333 \), then the system admits two positive fixed points. Figure~\ref{twofp} shows the curves corresponding to the first coordinate of the positive fixed points as a function of the parameter \( \beta \).

To verify that a Neimark-Sacker bifurcation can also occur in this case, we solve the system of equations \( \beta = \Psi(u) \) and \( q(u) = 1 \), obtaining the bifurcation value
\[
\beta_0 \approx 4.734145.
\]
Since \( \beta_0 \in \left( \Psi(\widehat{u}), \Psi(1) \right) \), there exist two positive fixed points at \( \beta = \beta_0 \). These fixed points are approximately given by
\[
E_1 \approx (0.4678,\, 0.2944), \quad E_2 \approx (0.6549,\, 0.2470).
\]
Moreover, at the fixed point \( E_1 \), the condition \( q(u_1) = 1 \) is satisfied. Additionally, \( q(u_1) > 1 \) for \( \beta > \beta_0 \) and \( q(u_1) < 1 \) for \( \beta < \beta_0 \), indicating a bifurcation at \( \beta_0 \).

We now compute the discriminating quantity \( \mathcal{L} \) at the point \( E_1 \).

The corresponding multipliers are
\[
\lambda_1 \approx 0.9499 - 0.3123i, \quad \lambda_2 \approx 0.9499 + 0.3123i,
\]
which are complex conjugates.

The normal form coefficients are computed as
\[
L_{20} \approx -0.1883 - 4.3521i, \quad L_{11} \approx -2.1160 - 9.0345i,
\]
\[
L_{02} \approx -1.9242 - 4.3348i, \quad L_{21} \approx -13.6802 + 78.0383i.
\]

The resulting discriminating quantity is
\[
\mathcal{L} \approx -90.9083 < 0.
\]

Therefore, by Theorem~\ref{bifurcation}, the system~\eqref{h12} undergoes a \textbf{Neimark--Sacker bifurcation}, and an \textbf{attracting} invariant closed curve emerges from the fixed point when \( \beta > \beta_0 \) (i.e., when \( \beta^* > 0 \)).

In Figure~\ref{fig2}, two trajectories of system~\eqref{h12} are presented for different values of \( \beta \). In panel~(a), an invariant closed curve is formed, while in panel~(b), the trajectory converges to the fixed point \( (1, 0) \).

\begin{figure}[h!]
    \centering
  \includegraphics[width=0.7\textwidth]{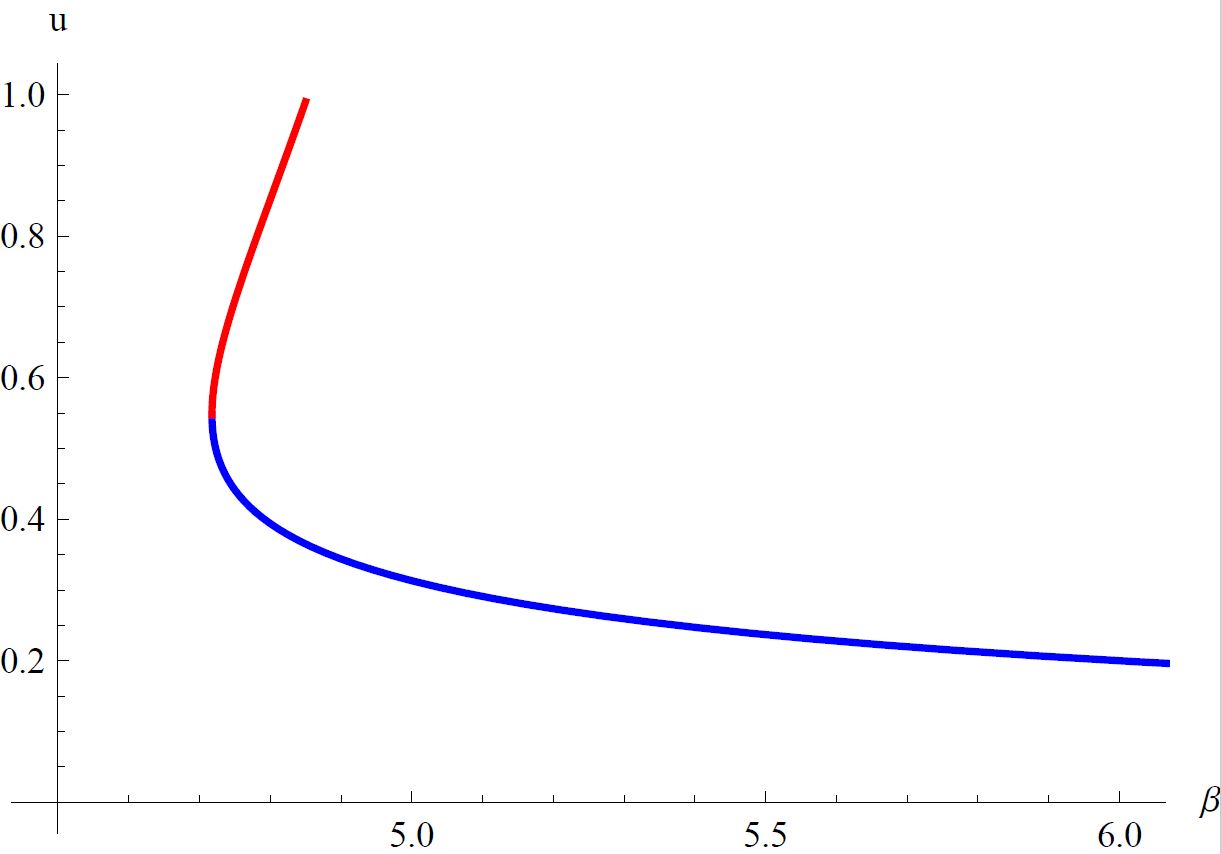}
    \caption{Bifurcation behaviour of the model (\ref{h12}) with parameters \( r = 0.5 \), \( \theta = 5 \), and \( \gamma = 0.2 \). The horizontal axis shows the bifurcation parameter \( \beta \), while the vertical axis represents the first coordinate \( u \) of the positive fixed points. The blue curve corresponds to the fixed point \( (u_1, v_1) \), and the red curve to the saddle point \( (u_2, v_2) \). For \( \beta \in(4.71762,4.85333) \), two distinct positive fixed points exist.}
    \label{twofp}
\end{figure}

\begin{figure}[h!]
    \centering
    \subfigure[\tiny$\beta=4.75, u_1\approx0.4413, q(u_1)\approx1.0775, (0.7, 0.3).$]{\includegraphics[width=0.45\textwidth]{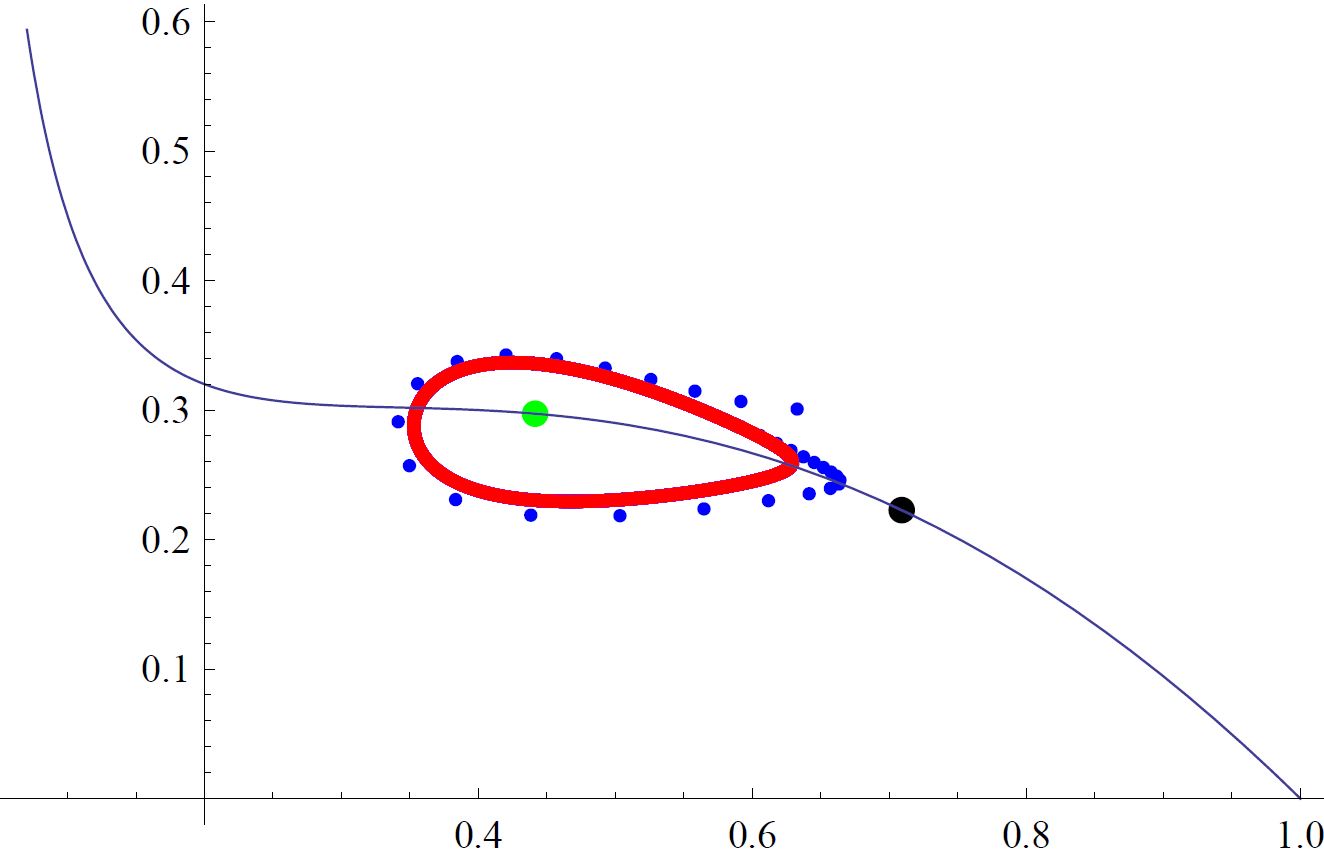}} \hspace{0.3in}
    \subfigure[\tiny$\beta=4.78, u_1\approx0.4094, q(u_1)\approx1.1807, (0.45, 0.3).$]{\includegraphics[width=0.45\textwidth]{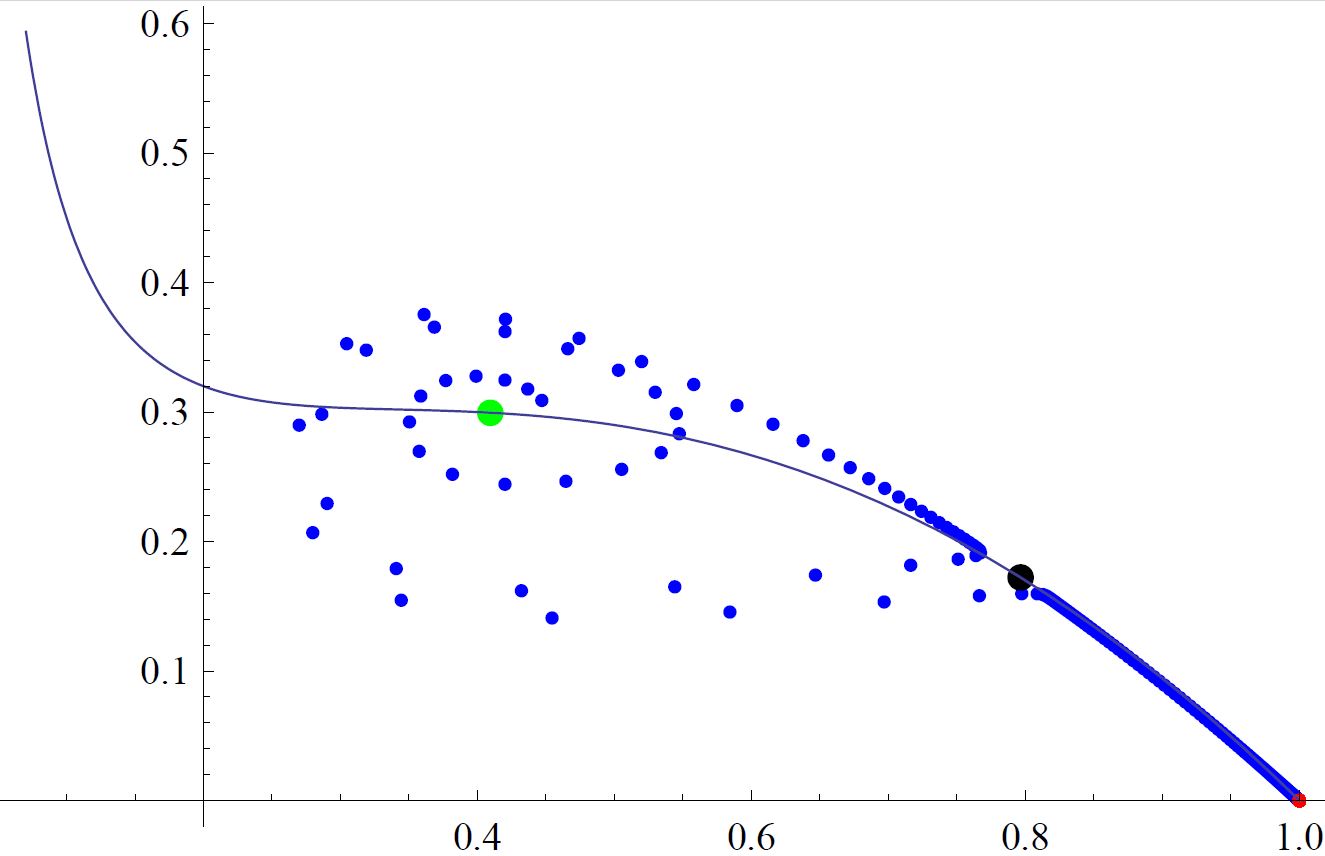}} \hspace{0.3in}
   \caption{Two trajectories of system~\eqref{h12} with parameters \( r = 0.5 \), \( \theta = 5 \), \( \gamma = 0.2 \), and \( n = 10{,}000 \). The green point denotes the fixed point \( E_1 \), while the black point corresponds to saddle point \( E_2 \). The blue curve represents the function \( \frac{(1 - x)(\gamma^2 + x^2)}{x} \), along which all positive fixed points lie.}
    \label{fig2}
\end{figure}

\section*{Conclusion}

In this work, we studied the dynamics of a discrete-time phytoplankton-zooplankton model with Holling type~III predation and type~II toxin release. Theorem~\ref{thm1} established conditions for the existence of positive fixed points, and Theorem~\ref{type} classified their stability. We found that whenever two positive fixed points \(E_1\) and \(E_2\) exist, \(E_2\) is always a saddle. Theorem~\ref{bifurcation} showed that system~\eqref{h12} undergoes a Neimark-Sacker bifurcation at \(E_1\) under suitable parameter conditions. Furthermore, Propositions~\ref{lim1}, \ref{lim2}, \ref{lim3}  proved that the boundary fixed point \( (1,0) \) is globally stable when no positive fixed point exists. These results reveal the system’s rich dynamical behavior and may inform future ecological modeling.

To illustrate the theory, we presented two examples. In Example~1, parameters were chosen so that the system has a single positive fixed point, leading to an attracting closed invariant curve. In Example~2, two positive fixed points were considered, one undergoing a Neimark-Sacker bifurcation and the other a saddle, highlighting the model’s broader dynamical possibilities.

For future work, we plan to explore more general functional responses and extend the model to include additional trophic groups, such as mixoplankton and bacteria, which are essential in marine ecosystems. We also aim to investigate other bifurcation types and derive rigorous analytical criteria for the emergence of chaotic attractors and other complex dynamical phenomena.



\section*{Declarations}

\noindent This study did not involve human participants or animals and therefore required no ethical approval; it received no financial support, and no data are applicable for availability.

\end{document}